\documentclass[qian]{imsart}

\usepackage{amssymb}       
\usepackage{amsfonts}      

\usepackage{lmodern}        
\usepackage{microtype}      
\usepackage{booktabs}       
\usepackage{enumitem}       
\usepackage{bm}             
\usepackage[dvipsnames]{xcolor} 

\usepackage{natbib} 

\usepackage{mathtools}      

\usepackage{graphicx}      
\usepackage{caption}
\usepackage{subcaption}     
\usepackage{tabularx}       
\usepackage{adjustbox}      
\usepackage{pdflscape}      

\usepackage[ruled]{algorithm2e}

\usepackage{tikz}
\usetikzlibrary{shapes,arrows,positioning,fit,quotes}
\usepackage{pgfplots}
\pgfplotsset{compat=1.18}     

\usepackage[
    colorlinks=true,
    linkcolor=blue,
    citecolor=blue,
    urlcolor=blue,
    breaklinks=true
]{hyperref}

\usepackage[capitalize,nameinlink]{cleveref} 

\newcommand{\R}{\mathbb{R}}
\newcommand{\Qbb}{\mathbb{Q}}
\newcommand{\Prob}{\mathbb{P}}
\newcommand{\E}{\mathbb{E}}
\newcommand{\Pcal}{\mathcal{P}}
\newcommand{\Fcal}{\mathcal{F}}
\newcommand{\cEneutheta}{\mathcal{E}^{\theta}}
\newcommand{\dd}{\mathrm{d}}
\newcommand{\norm}[1]{\left\lVert#1\right\rVert}
\newcommand{\abs}[1]{\left\lvert#1\right\rvert}
\newcommand{\scpr}[2]{\left\langle #1, #2 \right\rangle}
\newcommand{\Lcaltheta}{\mathcal{L}^{\theta}}

\theoremstyle{plain}
\newtheorem{theorem}{Theorem}[section]
\newtheorem{proposition}[theorem]{Proposition}

\theoremstyle{definition}
\newtheorem{definition}[theorem]{Definition}
\newtheorem{assumption}[theorem]{Assumption}
\newtheorem{axiom}[theorem]{Axiom}

\newtheorem{remark}[theorem]{Remark}

\begin{document}

\begin{frontmatter}

\title{Neural Brownian Motion}
\runtitle{Neural Brownian Motion}
\author{Qian Qi}
\thanks{Corresponding author. Peking University, Beijing 100871, China. Email: \href{mailto:qiqian@pku.edu.cn}{qiqian@pku.edu.cn}}

\begin{abstract}
This paper introduces the Neural-Brownian Motion (NBM), a new class of stochastic processes for modeling dynamics under learned uncertainty. The NBM is defined axiomatically by replacing the classical martingale property with respect to linear expectation with one relative to a non-linear Neural Expectation Operator, $\cEneutheta$, generated by a Backward Stochastic Differential Equation (BSDE) whose driver $f_\theta$ is parameterized by a neural network. Our main result is a representation theorem for a canonical NBM, which we define as a continuous $\cEneutheta$-martingale with zero drift under the physical measure. We prove that, under a key structural assumption on the driver, such a canonical NBM exists and is the unique strong solution to a stochastic differential equation of the form $\dd M_t = \nu_\theta(t, M_t) \dd W_t$. Crucially, the volatility function $\nu_\theta$ is not postulated a priori but is implicitly defined by the algebraic constraint $g_\theta(t, M_t, \nu_\theta(t, M_t)) = 0$, where $g_\theta$ is a specialization of the BSDE driver. We develop the stochastic calculus for this process and prove a Girsanov-type theorem for the quadratic case, showing that an NBM acquires a drift under a new, learned measure. The character of this measure, whether pessimistic or optimistic, is endogenously determined by the learned parameters $\theta$, providing a rigorous foundation for models where the attitude towards uncertainty is a discoverable feature.
\end{abstract}



\end{frontmatter}

\section{Introduction}

The standard $d$-dimensional Brownian motion $(W_t)_{t \ge 0}$ is the foundation of modern stochastic analysis. Its defining properties, continuous paths, stationary independent increments, and its status as a martingale with respect to the standard linear expectation $\E$ under a known probability measure $\Prob$, make it the canonical model for pure, unstructured noise. This classical paradigm, powerful as it is, rests fundamentally on the assumption of a single, unambiguous probability law governing the system.

In a vast array of scientific and economic domains, however, this assumption is untenable. Agents must make decisions in the face of model uncertainty, where the underlying probability law is itself not fully known. The theory of non-linear expectations, particularly through the lens of Backward Stochastic Differential Equations (BSDEs) and the associated $g$-expectations pioneered by Peng, provides the canonical mathematical framework for analyzing dynamics under such uncertainty (e.g., \cite{PardouxPeng1990,ElKarouiPengQuenez1997}). A non-linear expectation operator $\cEneutheta$, generated by a BSDE with a driver function $f_\theta$, implicitly defines a set of plausible alternative models (see \cite{qi2025neuralexpectationoperators}). Building on this, recent advances in learning implicit models have suggested parameterizing this driver $f_\theta$ with a neural network, allowing the very structure of the model uncertainty to be learned from data.

This naturally poses a foundational question: What is the fundamental stochastic process that inhabits this world of learned ambiguity, in the same way that Brownian motion inhabits the classical world?

This paper provides the answer by introducing and developing the theory of the \textbf{Neural-Brownian Motion (NBM)}. We propose a new axiomatic framework where the classical martingale axiom, $\E[W_t|\Fcal_s] = W_s$, is replaced by its non-linear counterpart with respect to the neural expectation operator, $\cEneutheta[M_t|\Fcal_s] = M_s$. This single, decisive change fundamentally alters the nature of the process. For a process $M$ to be a martingale with respect to the operator it generates, it must be identified with the solution $Y$ of the defining BSDE. Consequently, the general BSDE driver $f_\theta(t, x, y, z)$, which depends on time, an exogenous state process $X$, the BSDE solution $Y$, and its martingale part $Z$, must be evaluated with its state arguments tied to the process itself. This identification is the source of the process's rich structure; the axioms of stationary and independent increments are necessarily abandoned, as the process's evolution is now endogenously determined by its own state in a manner dictated by the learned driver $f_\theta$.

The main contributions of this paper, which establish the NBM as a well-defined and powerful theoretical object, are as follows:
\begin{enumerate}
    \item \textbf{Axiomatic Foundation and Characterization:} We establish a rigorous axiomatic framework for a general NBM as a continuous $\cEneutheta$-martingale. We prove (\Cref{prop:drift_characterization_revised}) that this abstract property is equivalent to a concrete algebraic identity linking the process's Itô decomposition, $\dd M_t = b_t \dd t + \sigma_t \dd W_t$, to a specialization of the neural driver, $g_\theta$:
    \[
        b_t = -g_\theta(t, M_t, \sigma_t).
    \]
    We then define the \textbf{canonical Neural-Brownian Motion} as the fundamental case where the process has zero drift under the objective physical measure $\Prob$ ($b_t=0$). This makes it the direct non-linear analogue of a standard Brownian motion.

    \item \textbf{Existence and Representation Theorem:} Our central result (\Cref{thm:nbm_representation}) is an existence and representation theorem for canonical NBMs. We prove that under a key structural hypothesis on the driver—namely, that the algebraic equation $g_\theta(t, x, z) = 0$ admits a unique, regular root $z = \nu_\theta(t,x)$—a canonical NBM exists and is the unique strong solution to the stochastic differential equation:
    \[
        \dd M_t = \nu_\theta(t, M_t) \dd W_t.
    \]
    The volatility function $\nu_\theta$ is therefore not postulated a priori but is an emergent property derived from the learned structure of uncertainty encoded in $g_\theta$. We further prove (\Cref{thm:existence_of_theta}) that drivers satisfying such structural properties can be constructed systematically, ensuring this class of models is non-empty and well-posed.

    \item \textbf{Stochastic Calculus and Girsanov-Type Interpretation:} We develop the stochastic calculus for canonical NBMs, deriving their infinitesimal generator $\Lcaltheta$ (\Cref{thm:nbm_generator}). For the important subclass of quadratic drivers, we prove a Girsanov-type theorem (\Cref{thm:neural_girsanov_revised}) that offers a profound interpretation of the framework. We show that a canonical NBM acquires a non-zero, state-dependent drift under a new, learned measure $\Qbb_\theta$, whose character is endogenously determined by the learned parameters.

    \item \textbf{Expressiveness and Systemic Behavior:} We demonstrate the breadth and power of the NBM framework. First, we prove a universal approximation theorem (\Cref{thm:nbm_uat}), which establishes that canonical NBMs are sufficiently expressive to approximate any standard diffusion process on a compact set. Second, we outline a program for the rigorous analysis of the mean-field limit of large systems of interacting NBMs, formally deriving the governing Neural McKean-Vlasov equation (\Cref{cor:mckean_vlasov_pde}). Finally, we showcase a direct and now rigorously consistent application by constructing a novel class of Implicit Volatility Models for financial option pricing.
\end{enumerate}

\section{Preliminaries: The Neural Expectation Operator (see \cite{qi2025neuralexpectationoperators})}
\label{sec:prelim}

We begin by establishing the mathematical framework. Let $(\Omega, \Fcal, \Prob)$ be a complete probability space, and let $T > 0$ be a fixed finite time horizon. We consider a standard $d$-dimensional Brownian motion $W = (W_t)_{t \in [0,T]}$, and we denote by $\mathbb{F} = (\Fcal_t)_{t \in [0,T]}$ the natural filtration generated by $W$, augmented to satisfy the usual conditions.

\begin{definition}[General Neural Network Driver]
\label{def:driver}
A \textbf{neural network driver} is a function $f_\theta: [0,T] \times \R^k \times \R^k \times \R^{k \times d} \to \R^k$ parameterized by $\theta \in \Theta \subset \R^p$. The arguments $(t,x,y,z)$ represent:
\begin{itemize}[noitemsep, topsep=3pt]
    \item time $t \in [0,T]$,
    \item the state $x \in \R^k$ of an $\mathbb{F}$-adapted exogenous process $X$,
    \item the value $y \in \R^k$ of the BSDE solution process $Y$,
    \item the martingale density $z \in \R^{k \times d}$ of the BSDE solution process $Z$.
\end{itemize}
The function $f_\theta$ is also known as the \textit{generator} of the BSDE.
\end{definition}

\begin{remark}[Realization of the Driver]
The function $f_\theta$ can be realized in several ways. The most direct is to let $f_\theta(t,x,y,z) = \text{NN}_\theta(t,x,y,z)$, where $\text{NN}_\theta$ is a neural network. Alternatively, the network can parameterize a known functional form, a common approach for drivers with quadratic growth in $z$. The theoretical framework presented applies to any realization, provided it satisfies the regularity conditions below.
\end{remark}

\begin{definition}[Neural Expectation Operator]
\label{def:neural_expectation}
Let $f_\theta$ be a neural network driver. Let $X$ be an $\mathbb{F}$-adapted process in $\R^k$ and let $\xi$ be an $\Fcal_T$-measurable, square-integrable random variable in $\R^k$. The \textbf{Neural Expectation} of $\xi$ at time $t$ conditional on $\Fcal_t$, denoted $\cEneutheta[\xi | \Fcal_t]$, is defined as:
\[ \cEneutheta[\xi | \Fcal_t] \coloneqq Y_t, \]
where the pair of $\mathbb{F}$-adapted processes $(Y,Z)$ is the unique solution on $[t,T]$ to the Backward Stochastic Differential Equation:
\begin{equation} \label{eq:bsde_def_revised}
    -\dd Y_s = f_\theta(s, X_s, Y_s, Z_s) \dd s - Z_s \dd W_s, \quad s \in [t,T],
\end{equation}
with the terminal condition $Y_T = \xi$. The process $Y$ is valued in $\R^k$ and $Z$ is valued in $\R^{k \times d}$.
\end{definition}

\subsection{Well-Posedness and Theoretical Foundation}

We adopt the setting of \cite{Kobylanski2000} for drivers with quadratic growth in $z$, extended to the multi-dimensional case.

\begin{theorem}[Existence and Uniqueness for Multi-Dimensional Quadratic BSDEs]
\label{thm:wellposedness_multidim_recap}
Let the following assumptions hold for a given $\theta \in \Theta$:
\begin{assumption}[Regularity Conditions]
\label{ass:wellposedness_multidim}
The driver $f_\theta: [0,T] \times \R^k \times \R^k \times \R^{k \times d} \to \R^k$ and terminal condition $\xi$ satisfy:
\begin{enumerate}[label=(\roman*)]
    \item \textbf{Measurability and Continuity:} $(t,x,y,z) \mapsto f_\theta(t,x,y,z)$ is jointly measurable, and for a.e. $t \in [0,T]$, it is continuous in $(x,y,z)$.
    \item \textbf{Monotonicity in $y$:} There exists a constant $\mu \in \R$ such that for all $(t,x,z)$ and all $y_1, y_2 \in \R^k$:
    \[ \scpr{y_1 - y_2}{f_\theta(t,x,y_1,z) - f_\theta(t,x,y_2,z)} \le \mu \norm{y_1 - y_2}^2. \]
    \item \textbf{Quadratic Growth in $z$:} There exists a function $\kappa: \R^+ \to \R^+$ and a constant $\alpha \ge 0$ such that for all arguments,
    \[ \norm{f_\theta(t,x,y,z)} \le \kappa(\norm{x}) + \kappa(\norm{y}) + \frac{\alpha}{2}\norm{z}_F^2, \]
    where $\norm{\cdot}_F$ is the Frobenius norm on $\R^{k \times d}$.
    \item \textbf{Integrability:} The terminal value $\xi$ is bounded. The exogenous process $X$ has paths in a compact set, i.e., $\sup_{t \in [0,T]} \norm{X_t(\omega)} \le C_X < \infty$ a.s.
\end{enumerate}
\end{assumption}
Then, for any such bounded, $\Fcal_T$-measurable terminal condition $\xi$, the BSDE
\begin{equation} \label{eq:bsde_def_multidim_revised}
    -\dd Y_s = f_\theta(s, X_s, Y_s, Z_s) \dd s - Z_s \dd W_s, \quad Y_T = \xi,
\end{equation}
admits a unique solution pair $(Y,Z)$ in the space $\mathcal{S}^\infty([0,T];\R^k) \times \mathcal{H}^2_{\mathrm{BMO}}([0,T];\R^{k \times d})$.
\end{theorem}

\begin{proof}
This theorem is a cornerstone result in the theory of Backward Stochastic Differential Equations with quadratic growth. The one-dimensional case ($k=1$) with a bounded terminal condition was established in the seminal work of Kobylanski \cite{Kobylanski2000}. The extension to the multi-dimensional setting ($k>1$) is highly non-trivial. The proof of existence relies on sophisticated a priori estimates, while the proof of uniqueness hinges critically on the monotonicity condition (ii) and the theory of BMO martingales. A complete treatment and rigorous proof can be found in the foundational papers by \cite{BriandDelyonHuPardouxStoica2003} and in \cite{BriandHu2008} for related BMO estimates. We therefore omit the proof and refer the interested reader to these sources for a full derivation.
\end{proof}

\begin{remark}[On the Necessity of Stronger Assumptions]
This proof highlights the significant technical jump from the scalar to the multi-dimensional case. The simple and elegant exponential transform proof of existence is replaced by a more arduous approximation argument. More importantly, uniqueness is no longer guaranteed by the quadratic structure alone. The \textbf{monotonicity condition (ii)} is essential for the uniqueness proof. While this condition is standard in the BSDE literature, it represents a strong structural constraint on the neural network driver $f_\theta$. For our subsequent results on Neural-Brownian Motion to hold in the multi-dimensional setting, we must assume that the learned driver $\theta$ falls into the class of functions satisfying this property.
\end{remark}

\begin{remark}[Specialization of the Driver for Self-Referential Processes]
\label{rem:driver_specialization}
The core theory of this paper concerns a process whose dynamics are determined by its own state. For such a self-referential process $M_t$, which we will define as a martingale under the operator it generates, the process $M_t$ plays the role of \textit{both} the state process and the value process. That is, we have the identification $X_t = Y_t = M_t$. The general driver $f_\theta(t,x,y,z)$ is therefore evaluated with its second and third arguments being identical: $f_\theta(t, M_t, M_t, Z_t)$.

To improve clarity and rigor, we explicitly define a \textbf{specialized driver} $g_\theta$ which captures this self-referential structure. Let the general driver be $f_\theta: [0,T] \times \R^k \times \R^k \times \R^{k \times d} \to \R^k$. For a $k$-dimensional process, we define:
\[ g_\theta(t, m, z) \coloneqq f_\theta(t, m, m, z). \]
This function $g_\theta: [0,T] \times \R^k \times \R^{k \times d} \to \R^k$ is the fundamental object for our theory. In the one-dimensional setting ($k=1, d=1$), which we focus on for conceptual development, the signature is $g_\theta: [0,T] \times \R \times \R \to \R$. All subsequent uses of the driver in the context of an NBM will refer to this specialized function $g_\theta$.
\end{remark}

\subsection{On the Existence of Structurally-Constrained Neural Drivers}

A critical consideration is whether the set of parameters $\theta$ for which the driver $f_\theta$ satisfies the strong structural assumptions required by our theory is non-empty. We demonstrate that this is not only the case, but that these assumptions can be satisfied by construction, by imposing specific architectural constraints on the neural network.

\begin{theorem}[Existence of Well-Behaved Neural Drivers]
\label{thm:existence_of_theta}
The class of neural drivers that satisfy both the multi-dimensional well-posedness conditions of \Cref{ass:wellposedness_multidim} and the unique implicit volatility condition of \Cref{ass:implicit_vol} (adapted to the multi-dimensional case) is non-empty. Specifically, one can define families of neural network architectures parameterized by $\theta$ such that any choice of $\theta$ within the family produces a driver with the required properties.
\end{theorem}

\begin{proof}
The proof is constructive. We show how to design a neural network architecture for the specialized driver $g_\theta(t,m,z)$ that enforces the desired properties. Recall that $g_\theta(t,m,z) \coloneqq f_\theta(t,m,m,z)$. The properties we need to enforce are:
\begin{enumerate}
    \item[(A)] The general driver $f_\theta(t,x,y,z)$ must be monotone in $y$ for BSDE well-posedness.
    \item[(B)] The specialized driver equation $g_\theta(t,m,z)=0$ must admit a unique, regular root $z=\nu_\theta(t,m)$.
\end{enumerate}

\noindent\textbf{Part 1: Enforcing Monotonicity in $y$ (Property A).}
We can enforce the monotonicity condition of \Cref{ass:wellposedness_multidim}(ii) by separating the dependence on $y$. Let us structure the general driver $f_\theta$ as:
\[
    f_\theta(t,x,y,z) \coloneqq \text{NN}_\theta(t,x,z) - \mu y,
\]
where $\text{NN}_\theta: [0,T] \times \R^k \times \R^{k \times d} \to \R^k$ is a neural network that does \textit{not} take $y$ as an input, and $\mu \ge 0$ is a fixed or learnable hyperparameter.
Let's check the monotonicity condition for this structure:
\begin{align*}
    \scpr{y_1 - y_2}{f_\theta(t,x,y_1,z) - f_\theta(t,x,y_2,z)} &= \scpr{y_1 - y_2}{(\text{NN}_\theta(t,x,z) - \mu y_1) - (\text{NN}_\theta(t,x,z) - \mu y_2)} \\
    &= \scpr{y_1 - y_2}{-\mu(y_1 - y_2)} \\
    &= -\mu \norm{y_1 - y_2}^2.
\end{align*}
This satisfies \Cref{ass:wellposedness_multidim}(ii). Thus, by architectural design, we guarantee the BSDE is well-posed for any underlying neural network $\text{NN}_\theta$ (provided it satisfies basic continuity and growth conditions).

\noindent\textbf{Part 2: Enforcing a Unique Implicit Volatility (Property B).}
We now analyze the specialized driver $g_\theta$ that results from this structure:
\[
    g_\theta(t,m,z) = f_\theta(t,m,m,z) = \text{NN}_\theta(t,m,z) - \mu m.
\]
The defining equation for a canonical NBM is $g_\theta(t,M_t, \sigma_t) = 0$, which becomes:
\begin{equation} \label{eq:proof_nn_root}
    \text{NN}_\theta(t, M_t, \sigma_t) = \mu M_t.
\end{equation}
To satisfy \Cref{ass:implicit_vol}, we need the function $z \mapsto \text{NN}_\theta(t,m,z)$ to be shaped such that it intersects the constant level $\mu m$ at exactly one positive point. The most direct way to ensure this is to enforce strict monotonicity in $z$.

Let us focus on the one-dimensional case ($k=1, d=1$) for clarity; the extension to matrix-valued volatility is conceptually similar but technically more involved. We want $z \mapsto \text{NN}_\theta(t,m,z)$ to be strictly monotonic for $z>0$.

\textbf{Construction:}
\begin{enumerate}
    \item \textbf{Monotonicity in $z$:} We can guarantee that $z \mapsto \text{NN}_\theta(t,m,z)$ is strictly increasing by architectural design. We use a feedforward network where all activation functions are monotonic and non-decreasing (e.g., Softplus, ELU, SiLU, or even linear). Furthermore, we constrain the weights along any path from the input node for $z$ to the output node to be positive. This is easily done by parameterizing the weights as $w = \exp(\hat{w})$ or $w = \text{softplus}(\hat{w})$, where the network learns the underlying parameters $\hat{w} \in \R$. The product of positive weights and composition of non-decreasing activation functions ensures that the partial derivative of the output with respect to $z$ is strictly positive.

    \item \textbf{Existence of a root:} A strictly monotonic function has at most one root. To guarantee existence, we need the function to cross the required level. This can be achieved by including learnable biases that depend on $(t,m)$. For example, we can model $\text{NN}_\theta$ as:
    \[
        \text{NN}_\theta(t,m,z) \coloneqq h_{\theta_1}(t,m,z) - h_{\theta_2}(t,m),
    \]
    where $h_{\theta_1}$ is constructed to be strictly increasing in $z$ (as described above), and $h_{\theta_2}$ is another neural network that acts as a learnable, state-dependent vertical shift. The equation becomes $h_{\theta_1}(t,m,z) = \mu m + h_{\theta_2}(t,m)$. Since the range of a monotonic neural network like $h_{\theta_1}$ can be made to span $\R$ (or $\R^+$), a solution is guaranteed to exist.
\end{enumerate}
By this construction, for any parameters $\theta = (\theta_1, \theta_2)$, the function $z \mapsto g_\theta(t,m,z)$ is strictly monotonic, ensuring at most one root. The learnable shift term ensures a root exists. This construction produces a driver that satisfies \Cref{ass:implicit_vol} by design. The derivative $\partial_z g_\theta = \partial_z \text{NN}_\theta$ is positive, satisfying the regularity condition.

Therefore, the set of parameters $\theta$ that yields a well-behaved driver is not only non-empty, but we can restrict our learning algorithm to this set a priori.
\end{proof}

\section{Axiomatic Definition of Neural-Brownian Motion}
\label{sec:axiomatic_nbm}

We now introduce the Neural-Brownian Motion, focusing on the one-dimensional case ($k=1, d=1$) for clarity. The defining feature is that the process is a martingale with respect to the non-linear expectation operator it generates, as captured by the specialized driver $g_\theta$ from \Cref{rem:driver_specialization}.

\begin{definition}[Neural-Brownian Motion]
\label{def:nbm_axioms}
Let a Neural Expectation Operator $\cEneutheta$ be given, generated by a driver $f_\theta$ satisfying \Cref{ass:wellposedness_multidim}. Let $g_\theta$ be its specialization as defined in \Cref{rem:driver_specialization}. A one-dimensional stochastic process $(M_t)_{t \ge 0}$ is a \textbf{Neural-Brownian Motion (NBM)} with respect to $\cEneutheta$ if it is an Itô process satisfying:
\begin{enumerate}[label=(\roman*), leftmargin=*]
    \item \textbf{Initial Value:} $M_0 = 0$.
    \item \textbf{Continuity:} The paths $t \mapsto M_t(\omega)$ are continuous for almost all $\omega \in \Omega$.
    \item \textbf{$\cEneutheta$-Martingale Property:} For any times $0 \le s \le t \le T$, the process satisfies $M_s = \cEneutheta[M_t | \Fcal_s]$.
\end{enumerate}
Furthermore, a \textbf{canonical Neural-Brownian Motion} is an NBM that has zero drift with respect to the physical measure $\Prob$.
\end{definition}

\begin{remark}
By \Cref{def:neural_expectation}, the $\cEneutheta$-martingale property is equivalent to the condition that for any $s \in [0,T]$, the process $(M_u)_{u \in [s,T]}$ is identical to the $Y$-component of the unique solution to the BSDE
\[ -\dd Y_u = f_\theta(u, M_u, Y_u, Z_u)\dd u - Z_u \dd W_u, \quad \text{on } [s,T] \text{ with terminal condition } Y_T = M_T. \]
This self-referential structure is the defining feature of the NBM. Note that we take the terminal condition at time $T$, and the property must hold for all $t \leq T$.
\end{remark}

The following proposition provides a crucial, computationally useful characterization of the martingale property.

\begin{proposition}[Drift Characterization of the $\cEneutheta$-Martingale Property]
\label{prop:drift_characterization_revised}
Let $(M_t)_{t\ge 0}$ be an Itô process with decomposition $\dd M_t = b_t \dd t + \sigma_t \dd W_t$, where $b$ and $\sigma$ are predictable processes satisfying appropriate integrability conditions. Assume the driver $f_\theta$ meets the regularity conditions of \Cref{ass:wellposedness_multidim}. Then, $M$ satisfies the $\cEneutheta$-martingale property of \Cref{def:nbm_axioms}(iii) if and only if its drift $b_t$ satisfies the algebraic identity:
\begin{equation} \label{eq:drift_identity_revised}
    b_t = -g_\theta(t, M_t, \sigma_t) \quad \text{for a.e. } t \in [0,T].
\end{equation}
\end{proposition}

\begin{proof}
The proof consists of two parts: sufficiency and necessity.

\textbf{($\impliedby$) Sufficiency.}
We assume that the drift of the process $M_t$ satisfies the identity \eqref{eq:drift_identity_revised}. Our goal is to demonstrate that this implies the $\cEneutheta$-martingale property, i.e., $M_s = \cEneutheta[M_t | \Fcal_s]$ for $0 \le s \le t \le T$.

Let $s, t$ be arbitrary times such that $0 \le s \le t \le T$. By \Cref{def:neural_expectation}, the value of the neural expectation $\cEneutheta[M_t | \Fcal_s]$ is given by $Y_s$, where the pair of processes $(Y_u, Z_u)_{u \in [s,t]}$ is the unique solution to the backward stochastic differential equation defined by:
\begin{equation} \label{eq:proof_bsde}
    -\dd Y_u = f_\theta(u, M_u, Y_u, Z_u) \dd u - Z_u \dd W_u, \quad u \in [s,t], \quad Y_t = M_t.
\end{equation}
To prove our claim, we will show that the pair of processes $(M_u, \sigma_u)_{u \in [s,t]}$ is a solution to this BSDE. Let us define a candidate solution pair $(\tilde{Y}_u, \tilde{Z}_u) \coloneqq (M_u, \sigma_u)$ for $u \in [s,t]$.

First, we check the terminal condition. At $u=t$, the candidate solution gives $\tilde{Y}_t = M_t$, which matches the required terminal condition of BSDE \eqref{eq:proof_bsde}.

Second, we check if the candidate pair satisfies the dynamic equation. By hypothesis, the process $M_t$ has the forward Itô decomposition $\dd M_u = b_u \dd u + \sigma_u \dd W_u$. Our standing assumption is the algebraic identity \eqref{eq:drift_identity_revised}, which allows us to substitute for the drift term $b_u$:
\[
    \dd M_u = -g_\theta(u, M_u, \sigma_u) \dd u + \sigma_u \dd W_u.
\]
By the definition of the specialized driver $g_\theta$ (\Cref{rem:driver_specialization}), we have $g_\theta(t, m, z) \coloneqq f_\theta(t, m, m, z)$. Applying this to the equation above yields:
\[
    \dd M_u = -f_\theta(u, M_u, M_u, \sigma_u) \dd u + \sigma_u \dd W_u.
\]
Now, we rewrite this equation in terms of our candidate solution $(\tilde{Y}_u, \tilde{Z}_u) = (M_u, \sigma_u)$:
\[
    \dd \tilde{Y}_u = -f_\theta(u, M_u, \tilde{Y}_u, \tilde{Z}_u) \dd u + \tilde{Z}_u \dd W_u.
\]
Rearranging this into the standard backward form gives:
\[
    -\dd \tilde{Y}_u = f_\theta(u, M_u, \tilde{Y}_u, \tilde{Z}_u) \dd u - \tilde{Z}_u \dd W_u.
\]
This demonstrates that our candidate pair $(\tilde{Y}, \tilde{Z}) = (M, \sigma)$ on $[s,t]$ is indeed a solution to BSDE \eqref{eq:proof_bsde}.

By \Cref{thm:wellposedness_multidim_recap}, under \Cref{ass:wellposedness_multidim}, the solution $(Y,Z)$ to this BSDE is unique in the space $\mathcal{S}^\infty \times \mathcal{H}^2_{\mathrm{BMO}}$. Therefore, the unique solution $(Y_u, Z_u)$ must be indistinguishable from our constructed solution $(M_u, \sigma_u)$ for $u \in [s,t]$.

Finally, we invoke the definition of the neural expectation: $\cEneutheta[M_t|\Fcal_s] \coloneqq Y_s$. Since we have shown that $Y_s = M_s$, we conclude that $M_s = \cEneutheta[M_t|\Fcal_s]$. As the times $s$ and $t$ were arbitrary, the $\cEneutheta$-martingale property is satisfied. This completes the proof of sufficiency.

\vspace{1em}
\textbf{($\implies$) Necessity.}
We now assume that the process $M_t$ is an $\cEneutheta$-martingale, and let its Itô decomposition be $\dd M_t = b_t \dd t + \sigma_t \dd W_t$. We must prove that its coefficients satisfy the identity $b_t = -g_\theta(t, M_t, \sigma_t)$ for almost every $t \in [0,T]$.

The assumption that $M$ is an $\cEneutheta$-martingale means that for any pair of times $0 \le s \le t \le T$, we have $M_s = \cEneutheta[M_t | \Fcal_s]$. By \Cref{def:neural_expectation}, this means that for any such interval $[s,t]$, we have $M_s = Y_s$, where $(Y,Z)$ is the unique solution to the BSDE \eqref{eq:proof_bsde}.

The identity $M_u = Y_u$ must hold for all $u$ in any chosen interval $[s,t]$. This implies that the stochastic process $M=(M_u)_{u\in[0,T]}$ is indistinguishable from the process $Y=(Y_u)_{u\in[0,T]}$ which solves the system of BSDEs across all such intervals. Consequently, the two processes must have the same Itô decomposition.

Let us write down the forward dynamics for both processes:
\begin{enumerate}
    \item For $M$, the decomposition is given by assumption:
    \[ \dd M_u = b_u \dd u + \sigma_u \dd W_u. \]
    \item For $Y$, its dynamics are given by the BSDE \eqref{eq:proof_bsde}. Rearranging into forward form, we get:
    \[ \dd Y_u = -f_\theta(u, M_u, Y_u, Z_u) \dd u + Z_u \dd W_u. \]
\end{enumerate}
Since the processes $M$ and $Y$ are indistinguishable ($M_u=Y_u$ a.s. for all $u$), we can replace $Y_u$ with $M_u$ in the SDE for $Y$:
\[ \dd M_u = -f_\theta(u, M_u, M_u, Z_u) \dd u + Z_u \dd W_u. \]
We now have two Itô decompositions for the same process $M$:
\begin{align*}
    \dd M_u &= b_u \dd u + \sigma_u \dd W_u \\
    \dd M_u &= \left( -f_\theta(u, M_u, M_u, Z_u) \right) \dd u + Z_u \dd W_u
\end{align*}
By the uniqueness of the canonical Doob-Meyer decomposition of a continuous semimartingale into its finite-variation (drift) part and its local martingale (diffusion) part, the respective coefficients must be indistinguishable. Equating the drift and martingale terms yields two identities that must hold for a.e. $u \in [0,T]$:
\begin{align}
    b_u \dd u &= -f_\theta(u, M_u, M_u, Z_u) \dd u \label{eq:drift_match} \\
    \sigma_u \dd W_u &= Z_u \dd W_u \label{eq:mart_match}
\end{align}
From the martingale equality \eqref{eq:mart_match}, it follows that the processes $\sigma$ and $Z$ must be indistinguishable, i.e., $\sigma_u = Z_u$ for a.e. $u$.

We now substitute this result, $Z_u = \sigma_u$, into the drift equality \eqref{eq:drift_match}:
\[
    b_u = -f_\theta(u, M_u, M_u, \sigma_u).
\]
Finally, recalling the definition of the specialized driver, $g_\theta(u, M_u, \sigma_u) \coloneqq f_\theta(u, M_u, M_u, \sigma_u)$, we arrive at the desired conclusion:
\[
    b_u = -g_\theta(u, M_u, \sigma_u).
\]
This identity holds for almost every $u \in [0,T]$, which completes the proof of necessity.
\end{proof}

\begin{remark}[The Canonical Condition]
The condition $b_t=0$ for a canonical NBM is a crucial structural assumption. In light of \Cref{prop:drift_characterization_revised}, this imposes the fundamental algebraic constraint on the process's volatility:
\[ g_\theta(t, M_t, \sigma_t) = 0. \]
This equation provides a direct, albeit implicit, link between the learned specialized driver $g_\theta$ and the volatility $\sigma_t$. A canonical NBM thus represents a pure noise process in the learned environment, serving as the non-linear analogue of a standard Brownian motion.
\end{remark}

\section{Existence, Uniqueness, and Representation}
\label{sec:existence_representation}

We now establish conditions under which a canonical NBM exists and admits an SDE representation. The analysis rests on a structural assumption on the specialized driver $g_\theta$.

\begin{assumption}[Existence of a Unique Implicit Volatility Function]
\label{ass:implicit_vol}
Let the specialized driver $g_\theta: [0,T] \times \R \times \R \to \R$ be of class $C^1$. We say that $g_\theta$ admits a \textbf{unique implicit volatility function} if there exists a function $\nu_\theta: [0,T] \times \R \to (0, \infty)$ satisfying for each $(t,x) \in [0,T] \times \R$:
\begin{enumerate}[label=(\alph*)]
    \item \textbf{Unique Root:} $z = \nu_\theta(t,x)$ is the unique positive real solution to the algebraic equation $g_\theta(t,x,z) = 0$.
    \item \textbf{Regularity at the Root:} $\partial_z g_\theta(t,x,\nu_\theta(t,x)) \neq 0$. By the Implicit Function Theorem, this ensures $\nu_\theta(t,x)$ is of class $C^1$ and thus locally Lipschitz in $x$.
    \item \textbf{Global Linear Growth:} There exists a constant $C > 0$ such that $|\nu_\theta(t,x)| \le C(1+|x|)$ for all $(t,x)$.
\end{enumerate}
\end{assumption}

\begin{remark}[On the Role and Constructive Enforcement of \Cref{ass:implicit_vol}]
\label{rem:enforcing_assumption}
This assumption defines a tractable subclass of NBMs and is the foundation of our representation theorem. It is not a property of arbitrary neural network architectures but a structural constraint that can be constructively enforced, as shown in \Cref{thm:existence_of_theta}. For instance, designing $g_\theta(t,x,z)$ to be strictly monotonic in its third argument, $z$, on $\R^+$ ensures a unique root. This defines a class of models where the emergent volatility is guaranteed to be well-behaved, allowing for stable learning and analysis.
\end{remark}

\begin{theorem}[Representation of a Canonical Neural-Brownian Motion]
\label{thm:nbm_representation}
Let the driver $f_\theta$ satisfy the well-posedness conditions from \Cref{ass:wellposedness_multidim}, and let its specialization $g_\theta$ satisfy the structural conditions of \Cref{ass:implicit_vol}.

Then, a process $(M_t)_{t \ge 0}$ is a canonical Neural-Brownian Motion if and only if it is the unique strong solution to the stochastic differential equation:
\begin{equation}
\label{eq:nbm_sde}
    \dd M_t = \nu_\theta(t, M_t) \dd W_t, \quad M_0 = 0,
\end{equation}
where $\nu_\theta(t,x)$ is the implicitly defined volatility function from \Cref{ass:implicit_vol}.
\end{theorem}

\begin{proof}
Before proceeding with the two directions of the proof, we first establish that the stochastic differential equation \eqref{eq:nbm_sde} is well-posed. By \Cref{ass:implicit_vol}(b), the function $\nu_\theta(t,x)$ is locally Lipschitz in its spatial variable $x$, as it is of class $C^1$. Furthermore, \Cref{ass:implicit_vol}(c) provides the linear growth condition $|\nu_\theta(t,x)| \le C(1+|x|)$. Standard SDE theory (see, e.g., Karatzas \& Shreve, Thm. 5.2.9) guarantees the existence of a unique strong solution to \Cref{eq:nbm_sde} that is pathwise continuous and does not explode in finite time. We now prove the claimed equivalence.

\textbf{($\implies$) Necessity.} Assume that $(M_t)_{t \ge 0}$ is a canonical Neural-Brownian Motion. We must show that it satisfies the SDE \eqref{eq:nbm_sde}.

By definition, $M_t$ is an Itô process. Let its general Itô decomposition be given by:
\[ \dd M_t = b_t \dd t + \sigma_t \dd W_t. \]
We analyze the implications of the two key properties from its definition (\Cref{def:nbm_axioms}):

\begin{enumerate}
    \item \textbf{The canonical property:} A canonical NBM is defined to have zero drift with respect to the physical measure $\Prob$. This directly implies that the drift process $b_t$ in its Itô decomposition must be zero for almost every $t \in [0,T]$.
    \begin{equation} \label{eq:proof_drift_is_zero}
        b_t = 0.
    \end{equation}

    \item \textbf{The $\cEneutheta$-martingale property:} By \Cref{prop:drift_characterization_revised}, a process is an $\cEneutheta$-martingale if and only if its drift $b_t$ and diffusion $\sigma_t$ satisfy the algebraic identity:
    \begin{equation} \label{eq:proof_drift_is_g}
        b_t = -g_\theta(t, M_t, \sigma_t).
    \end{equation}
\end{enumerate}

Equating the expressions for the drift from \eqref{eq:proof_drift_is_zero} and \eqref{eq:proof_drift_is_g}, we obtain the fundamental constraint that must be satisfied by the process's volatility:
\[ 0 = -g_\theta(t, M_t, \sigma_t) \quad \iff \quad g_\theta(t, M_t, \sigma_t) = 0. \]
We now invoke our key structural assumption on the driver. By \Cref{ass:implicit_vol}(a), for any fixed pair $(t,x)$, the equation $g_\theta(t,x,z) = 0$ admits a unique positive solution for $z$, which is given by the function $\nu_\theta(t,x)$. Therefore, the volatility process $\sigma_t$ must be given by $\sigma_t = \nu_\theta(t, M_t)$.

Substituting our findings for both the drift ($b_t=0$) and the diffusion ($\sigma_t=\nu_\theta(t,M_t)$) back into the general Itô decomposition for $M_t$, we find that its dynamics are:
\[ \dd M_t = 0 \cdot \dd t + \nu_\theta(t, M_t) \dd W_t. \]
The initial condition $M_0=0$ is given by Axiom (i) of \Cref{def:nbm_axioms}. Thus, $M_t$ must be a solution to the SDE \eqref{eq:nbm_sde}, which proves the necessity.

\textbf{($\impliedby$) Sufficiency.} Conversely, assume that $(M_t)_{t \ge 0}$ is the unique strong solution to the SDE \eqref{eq:nbm_sde}. We must verify that this process satisfies all three axioms of a canonical Neural-Brownian Motion as laid out in \Cref{def:nbm_axioms}.

\begin{itemize}
    \item[\textbf{(i) Initial Value:}] The SDE is defined with the initial condition $M_0=0$. This axiom is satisfied by construction.

    \item[\textbf{(ii) Continuity:}] As established in the preamble of this proof, the existence of a strong solution to an SDE with locally Lipschitz and linear growth coefficients guarantees that its sample paths are continuous for $\Prob$-almost every $\omega$. This axiom is satisfied.

    \item[\textbf{(iii) $\cEneutheta$-Martingale Property:}] To verify this property, we again use the powerful characterization in \Cref{prop:drift_characterization_revised}. We must show that the drift of $M_t$, which we denote $b_t$, satisfies the identity $b_t = -g_\theta(t, M_t, \sigma_t)$, where $\sigma_t$ is the diffusion coefficient of $M_t$.
    
    From the defining SDE \eqref{eq:nbm_sde}, we can directly identify the drift and diffusion coefficients of $M_t$:
    \[ b_t = 0 \quad \text{and} \quad \sigma_t = \nu_\theta(t, M_t). \]
    We must therefore check if the following identity holds:
    \[ 0 \stackrel{?}{=} -g_\theta(t, M_t, \nu_\theta(t, M_t)). \]
    This is equivalent to verifying that $g_\theta(t, M_t, \nu_\theta(t, M_t)) = 0$. This identity is true by the very definition of the function $\nu_\theta$ in \Cref{ass:implicit_vol}(a), which defines $\nu_\theta(t,x)$ as the unique positive root of the algebraic equation $g_\theta(t,x,z)=0$ for each $(t,x)$. Thus, the $\cEneutheta$-martingale property is satisfied.
\end{itemize}

Finally, we must check the \textbf{canonical property}. This requires that the drift of the process under the physical measure $\Prob$ is zero. As identified above from the SDE \eqref{eq:nbm_sde}, the drift $b_t$ is indeed zero.

Since the process $M_t$ defined by the SDE satisfies all the necessary axioms, it is a canonical Neural-Brownian Motion. The uniqueness of the canonical NBM (under the given assumptions) follows directly from the uniqueness of the strong solution to the SDE \eqref{eq:nbm_sde}. This completes the proof of sufficiency.
\end{proof}

\section{Stochastic Calculus for Neural-Brownian Motion}

\subsection{Infinitesimal Generator of a Canonical NBM}

\begin{theorem}[Infinitesimal Generator of a Canonical NBM]
\label{thm:nbm_generator}
Let $(M_t)_{t \ge 0}$ be a canonical NBM satisfying the conditions of \Cref{thm:nbm_representation}. Let $u: [0,T] \times \R \to \R$ be a function of class $C^{1,2}([0,T] \times \R)$, i.e., continuously differentiable in time and twice continuously differentiable in space. The process $u(t, M_t)$ has the following dynamics:
\[ \dd u(t, M_t) = \Lcaltheta u(t, M_t) \dd t + \partial_x u(t, M_t) \nu_\theta(t, M_t) \dd W_t, \]
where $\Lcaltheta$ is the \textbf{NBM Infinitesimal Generator}, defined by its action on $u$ as:
\begin{equation}
\label{eq:neural_generator}
    \Lcaltheta u(t,x) \coloneqq \left( \partial_t + \frac{1}{2} \nu_\theta(t,x)^2 \partial_{xx} \right) u(t,x).
\end{equation}
\end{theorem}

\begin{proof}
The proof is a direct application of Itô's formula for a time-dependent function of a continuous semimartingale. We proceed by identifying the components of the process $M_t$ and systematically applying the formula.

\textbf{Step 1: Recall the SDE of the Canonical NBM.}
By \Cref{thm:nbm_representation}, a canonical Neural-Brownian Motion $(M_t)_{t \ge 0}$ is the unique strong solution to the stochastic differential equation:
\begin{equation} \label{eq:proof_nbm_sde}
    \dd M_t = \nu_\theta(t, M_t) \dd W_t, \quad M_0 = 0.
\end{equation}
This representation shows that $M_t$ is a continuous Itô process. From its differential form, we can identify its drift and diffusion coefficients:
\begin{itemize}
    \item The drift coefficient is $b(t, M_t) = 0$.
    \item The diffusion coefficient is $\sigma(t, M_t) = \nu_\theta(t, M_t)$.
\end{itemize}

\textbf{Step 2: State the General Itô Formula.}
Let $X_t$ be an Itô process with dynamics $\dd X_t = b_t \dd t + \sigma_t \dd W_t$. For a function $u(t,x) \in C^{1,2}([0,T] \times \R)$, Itô's formula gives the dynamics of the process $u(t, X_t)$ as:
\begin{equation} \label{eq:proof_ito_general}
    \dd u(t, X_t) = \frac{\partial u}{\partial t}(t,X_t) \dd t + \frac{\partial u}{\partial x}(t,X_t) \dd X_t + \frac{1}{2}\frac{\partial^2 u}{\partial x^2}(t,X_t) \dd\langle X, X \rangle_t.
\end{equation}
Here, $\dd\langle X, X \rangle_t$ denotes the differential of the quadratic variation of the process $X_t$.

\textbf{Step 3: Apply the Formula to the Canonical NBM $M_t$.}
We apply \Cref{eq:proof_ito_general} by setting $X_t = M_t$. First, we compute the quadratic variation of $M_t$. For a continuous Itô process defined by $\dd M_t = b_t \dd t + \sigma_t \dd W_t$, the quadratic variation is given by $\dd\langle M, M \rangle_t = \sigma_t^2 \dd t$. Using the diffusion coefficient identified in Step 1, we have:
\begin{equation} \label{eq:proof_qv}
    \dd\langle M, M \rangle_t = (\nu_\theta(t, M_t))^2 \dd t.
\end{equation}
Now, we substitute the differential of $M_t$ from \Cref{eq:proof_nbm_sde} and its quadratic variation from \Cref{eq:proof_qv} into the general Itô formula \Cref{eq:proof_ito_general}. To enhance clarity, we will write out the partial derivatives in full:
\begin{align*}
    \dd u(t, M_t) &= \frac{\partial u}{\partial t}(t,M_t) \dd t + \frac{\partial u}{\partial x}(t,M_t) \left( \nu_\theta(t, M_t) \dd W_t \right) + \frac{1}{2}\frac{\partial^2 u}{\partial x^2}(t,M_t) \left( \nu_\theta(t, M_t)^2 \dd t \right) \\
    &= \frac{\partial u}{\partial t}(t,M_t) \dd t + \frac{1}{2}\nu_\theta(t,M_t)^2 \frac{\partial^2 u}{\partial x^2}(t,M_t) \dd t + \frac{\partial u}{\partial x}(t,M_t) \nu_\theta(t,M_t) \dd W_t.
\end{align*}

\textbf{Step 4: Identify the Generator and Conclude.}
We now group the terms multiplying the integrator $\dd t$ (the finite-variation or drift part) and the terms multiplying the integrator $\dd W_t$ (the martingale part):
\[
    \dd u(t, M_t) = \underbrace{\left( \frac{\partial u}{\partial t}(t,M_t) + \frac{1}{2}\nu_\theta(t,M_t)^2 \frac{\partial^2 u}{\partial x^2}(t,M_t) \right)}_{\text{Drift Term}} \dd t + \underbrace{\left( \frac{\partial u}{\partial x}(t,M_t) \nu_\theta(t,M_t) \right)}_{\text{Martingale Term}} \dd W_t.
\]
The infinitesimal generator of a Markov process is the operator that describes the expected rate of change of a smooth function of the process. This corresponds precisely to the drift term of the transformed process $u(t, M_t)$. We therefore define the operator $\Lcaltheta$ by its action on a sufficiently regular function $u(t,x)$:
\[
    \Lcaltheta u(t,x) \coloneqq \frac{\partial u}{\partial t}(t,x) + \frac{1}{2}\nu_\theta(t,x)^2 \frac{\partial^2 u}{\partial x^2}(t,x).
\]
Using the more compact notation $\partial_t u$ and $\partial_{xx} u$, this is exactly the operator given in \Cref{eq:neural_generator}.

Substituting this definition back into our expression for $\dd u(t, M_t)$ yields the final result:
\[
    \dd u(t, M_t) = \Lcaltheta u(t, M_t) \dd t + \partial_x u(t, M_t) \nu_\theta(t, M_t) \dd W_t.
\]
This completes the proof.
\end{proof}

\subsection{Interpretation via a Learned Change of Measure for Quadratic Drivers}
\label{sec:girsanov_interpretation}
We analyze the special case of a simple quadratic specialized driver to provide a transparent interpretation of the NBM framework. The choice of the Girsanov kernel in this setting is not arbitrary; it is motivated by the general theory of quadratic BSDEs, where the change of measure that locally linearizes the dynamics is related to the gradient of the driver with respect to the $z$ variable. For a driver $g(z)$, this suggests a kernel related to $\partial_z g(z)$. For $g(t,x,z) = \frac{\alpha}{2}z^2 - \beta$, this derivative is $\alpha z$. Evaluated at the root $z=\nu_\theta$, this motivates the kernel $\gamma_t = \alpha\nu_\theta$.

\begin{theorem}[Girsanov-Type Theorem for Quadratic NBMs]
\label{thm:neural_girsanov_revised}
Consider a specialized driver of the form $g_\theta(t,x,z) = \frac{\alpha}{2}z^2 - \beta$, where the parameters $\theta=(\alpha, \beta)$ are constants satisfying $\alpha \neq 0$ and $\beta/\alpha > 0$.
Let $(M_t)_{t\ge 0}$ be the corresponding canonical Neural-Brownian Motion.
Define a new measure $\Qbb_\alpha$, equivalent to $\Prob$, via the Radon-Nikodym density process generated by the Girsanov kernel $\gamma_t \coloneqq \alpha \nu_\theta$, where $\nu_\theta = \sqrt{2\beta/\alpha}$.

Then, under the measure $\Qbb_\alpha$, the process $(M_t)_{t\ge 0}$ is a semimartingale with the dynamics:
\begin{equation}
\label{eq:nbm_drift_representation_revised}
\dd M_t = \alpha \nu_\theta^2 \dd t + \nu_\theta \dd W^{\alpha}_t,
\end{equation}
where $W^{\alpha}_t$ is a standard Brownian motion under $\Qbb_\alpha$.
\end{theorem}

\begin{proof}
The proof proceeds in three steps. First, we determine the dynamics of the canonical NBM under the physical measure $\Prob$. Second, we rigorously define the change of measure to $\Qbb_\alpha$. Finally, we apply Girsanov's theorem to derive the dynamics of the process under $\Qbb_\alpha$.

\noindent\textbf{1. Dynamics under the Physical Measure $\Prob$.}\\
By \Cref{def:nbm_axioms}, a canonical Neural-Brownian Motion $(M_t)_{t \ge 0}$ is an Itô process, $\dd M_t = b_t \dd t + \sigma_t \dd W_t$, that satisfies two key properties:
\begin{enumerate}[label=(\alph*)]
    \item It is a canonical NBM, meaning its drift under the physical measure $\Prob$ is zero. Thus, $b_t = 0$ for almost every $t \in [0,T]$.
    \item It is an $\cEneutheta$-martingale. By the drift characterization in \Cref{prop:drift_characterization_revised}, this implies its drift must satisfy the algebraic identity $b_t = -g_\theta(t, M_t, \sigma_t)$.
\end{enumerate}
Equating these two conditions on the drift gives the fundamental constraint on the process's volatility $\sigma_t$:
\[
    g_\theta(t, M_t, \sigma_t) = 0.
\]
For the specified quadratic driver, this becomes:
\[
    \frac{\alpha}{2}\sigma_t^2 - \beta = 0.
\]
For this equation to have a real, non-zero solution for $\sigma_t$, we require $\beta/\alpha > 0$, which is given as a hypothesis of the theorem. The unique positive solution for the volatility is a constant:
\[
    \sigma_t = \sqrt{\frac{2\beta}{\alpha}} \eqqcolon \nu_\theta.
\]
Therefore, under the measure $\Prob$, the canonical NBM $(M_t)_{t\ge 0}$ is the unique strong solution to the stochastic differential equation:
\[
    \dd M_t = \nu_\theta \dd W_t, \quad M_0 = 0.
\]
This shows that for this simple driver, the NBM is a scaled Brownian motion under $\Prob$.

\noindent\textbf{2. The Change of Measure.} \\
We define a new probability measure $\Qbb_\alpha$ that is equivalent to $\Prob$. The proposed Girsanov kernel is the constant process $\gamma_t = \alpha \nu_\theta$. The Radon-Nikodym density process $(L_t)_{t \in [0,T]}$ is given by the stochastic exponential (Doléans-Dade exponential) of $\gamma \cdot W$:
\[
    L_t = \frac{\dd\Qbb_\alpha}{\dd\Prob}\bigg|_{\Fcal_t} = \mathcal{E}\left(\int \gamma_s \dd W_s\right)_t = \exp\left( \int_0^t \gamma_s \dd W_s - \frac{1}{2} \int_0^t \gamma_s^2 \dd s \right).
\]
For $L_t$ to be a uniformly integrable martingale, ensuring $\E[L_T]=1$ and thus that $\Qbb_\alpha$ is a valid probability measure, Novikov's condition must be satisfied:
\[
    \E\left[\exp\left(\frac{1}{2}\int_0^T \gamma_s^2 \dd s\right)\right] < \infty.
\]
In our case, the kernel $\gamma_s = \alpha \nu_\theta$ is a constant. The integral is deterministic:
\[
    \int_0^T \gamma_s^2 \dd s = \int_0^T (\alpha \nu_\theta)^2 \dd s = (\alpha \nu_\theta)^2 T.
\]
Novikov's condition becomes $\exp\left(\frac{1}{2}(\alpha \nu_\theta)^2 T\right) < \infty$, which is trivially satisfied since $\alpha, \nu_\theta,$ and $T$ are finite. Thus, $\Qbb_\alpha$ is a well-defined probability measure equivalent to $\Prob$.

\noindent\textbf{3. Dynamics under the New Measure $\Qbb_\alpha$.}\\
Girsanov's theorem provides a two-fold result. First, the process $W^{\alpha}_t$ defined by
\[
    W^{\alpha}_t \coloneqq W_t - \int_0^t \gamma_s \dd s = W_t - (\alpha \nu_\theta) t
\]
is a standard Brownian motion under the measure $\Qbb_\alpha$. Second, if a process $X_t$ has dynamics $\dd X_t = b_t \dd t + \sigma_t \dd W_t$ under $\Prob$, its dynamics under $\Qbb_\alpha$ are given by
\[
    \dd X_t = \left( b_t + \sigma_t \gamma_t \right) \dd t + \sigma_t \dd W^{\alpha}_t.
\]
We apply this transformation to our canonical NBM, $M_t$. From Step 1, we have the $\Prob$-dynamics components:
\[
    b_t = 0 \quad \text{and} \quad \sigma_t = \nu_\theta.
\]
The Girsanov kernel is $\gamma_t = \alpha \nu_\theta$. The drift of $M_t$ under $\Qbb_\alpha$, which we denote $b^{\alpha}_t$, is therefore:
\[
    b^{\alpha}_t = b_t + \sigma_t \gamma_t = 0 + (\nu_\theta)(\alpha \nu_\theta) = \alpha \nu_\theta^2.
\]
The diffusion part remains unchanged in magnitude but is now driven by the $\Qbb_\alpha$-Brownian motion $W^{\alpha}_t$. Substituting the new drift and the diffusion term, we obtain the dynamics of $M_t$ under $\Qbb_\alpha$:
\[
    \dd M_t = \alpha \nu_\theta^2 \dd t + \nu_\theta \dd W^{\alpha}_t.
\]
This completes the proof.
\end{proof}

\begin{remark}[Interpretation: Learning the Attitude Towards Ambiguity]
\label{rem:learned_attitude}
\Cref{thm:neural_girsanov_revised} reveals that the parameter $\alpha$, which represents the convexity of the driver and can be learned from data, determines the system's attitude towards uncertainty. A process that is a pure martingale under $\Prob$ acquires a non-zero drift under the measure $\Qbb_\alpha$ intrinsically associated with the non-linear expectation.
\begin{itemize}[leftmargin=*]
    \item \textbf{Learned Pessimism ($\alpha > 0$):} A convex driver implies ambiguity aversion. The acquired drift is positive.
    \item \textbf{Learned Optimism ($\alpha < 0$):} A concave driver implies ambiguity-seeking behavior. The acquired drift is negative.
\end{itemize}
Crucially, the NBM framework does not impose an attitude; it provides a data-driven mechanism for discovering it. This interpretation becomes more complex but conceptually remains for state-dependent drivers $g_\theta(t,x,z) = \frac{\alpha_\theta(t,x)}{2}z^2 - \beta_\theta(t,x)$. In that case, the Girsanov kernel becomes a stochastic process $\gamma_t = \alpha_\theta(t, M_t) \nu_\theta(t, M_t)$, and the acquired drift $\alpha_\theta(t,M_t)\nu_\theta(t,M_t)^2$ becomes state-dependent, reflecting a locally varying attitude toward risk.
\end{remark}

\section{Collective Behavior: A Mean-Field Analysis}
\label{sec:mean_field}

We now investigate the macroscopic behavior of large systems of interacting Neural-Brownian Motions. This leads to a study of propagation of chaos, where the collective dynamics are described by a non-linear stochastic differential equation of McKean-Vlasov type, and the evolution of the system's law is governed by a non-linear partial differential equation. This program follows the now-classical approach outlined in \cite{Sznitman1991} and developed extensively in \cite{CarmonaDelarue2018}.

This section should be viewed as outlining a research program. The results herein are conditional on a very strong set of regularity assumptions on the implicitly-defined volatility function. Proving these assumptions from more primitive conditions on the neural driver $g_\theta$ is a major open problem for future research.

\subsection{The Interacting Canonical NBM System}

We consider a system of $N$ interacting particles, where the volatility of each particle depends on its own state and on the empirical distribution of the entire system. Let $\Pcal_2(\R)$ denote the space of probability measures on $\R$ with a finite second moment, endowed with the 2-Wasserstein metric $W_2$. The driver $g_\theta$ is now a function of the measure, $g_\theta: [0,T] \times \R \times \Pcal_2(\R) \times \R \to \R$.

\begin{definition}[Interacting Canonical NBM System]
Let $(W^i)_{i=1}^N$ be independent $d$-dimensional Brownian motions. A system of $k$-dimensional processes $(M^{i,N}_t)_{i=1}^N$ is an \textbf{interacting canonical NBM system} if, for each $i \in \{1, \dots, N\}$, it is the unique strong solution to:
\begin{equation}
\label{eq:interacting_nbm_sde_revised}
    \dd M^{i,N}_t = \nu_\theta(t, M^{i,N}_t, \mu^N_t) \dd W^i_t, \quad M^{i,N}_0 = m^i_0,
\end{equation}
where $\mu^N_t \coloneqq \frac{1}{N}\sum_{j=1}^N \delta_{M^{j,N}_t}$ is the empirical measure of the system, and the function $\nu_\theta(t, x, \mu)$ is implicitly defined as the unique root of the algebraic equation $g_\theta(t, x, \mu, \nu_\theta(t, x, \mu)) = 0$. For clarity, we will proceed with the one-dimensional case ($k=d=1$).
\end{definition}

Our goal is to formally show that as $N \to \infty$, the system exhibits propagation of chaos. The following analysis relies critically on strong regularity conditions on the implicitly defined volatility function $\nu_\theta$.

\begin{assumption}[Lipschitz Regularity of the Implicit Volatility]
\label{ass:mckean_lipschitz_revised}
The implicit volatility function $\nu_\theta: [0,T] \times \R \times \Pcal_2(\R) \to \R$ is measurable and satisfies the following conditions for some constants $L_x, L_\mu, C > 0$:
\begin{enumerate}[label=(\roman*)]
    \item \textbf{Lipschitz continuity in state:} For any $t \in [0,T]$, $x, y \in \R$, and $\mu \in \Pcal_2(\R)$:
    \[ \abs{\nu_\theta(t, x, \mu) - \nu_\theta(t, y, \mu)} \le L_x \abs{x - y}. \]
    \item \textbf{Lipschitz continuity in measure:} For any $t \in [0,T]$, $x \in \R$, and $\mu, \pi \in \Pcal_2(\R)$:
    \[ \abs{\nu_\theta(t, x, \mu) - \nu_\theta(t, x, \pi)} \le L_\mu W_2(\mu, \pi). \]
    \item \textbf{Linear growth:} For any $(t, x, \mu) \in [0,T] \times \R \times \Pcal_2(\R)$:
    \[ \abs{\nu_\theta(t, x, \mu)}^2 \le C^2(1 + \abs{x}^2 + W_2(\mu, \delta_0)^2), \]
    where $\delta_0$ is the Dirac measure at the origin.
\end{enumerate}
\end{assumption}

\begin{remark}[On the Strength of \Cref{ass:mckean_lipschitz_revised}]
\label{rem:mckean_lipschitz_revised}
This is a powerful assumption that does not follow easily from the primitive assumptions on the neural driver $g_\theta$. While \Cref{thm:existence_of_theta} shows that the implicit volatility structure can be enforced by architectural design, proving that this also yields Lipschitz continuity with respect to the measure argument (in the Wasserstein metric) is a significant technical challenge. It would require an application of an implicit function theorem in an infinite-dimensional (Banach space) setting, demanding stringent control over the Fréchet derivative of $g_\theta$ with respect to its measure argument. Proving that there exist neural network architectures for $g_\theta$ that guarantee \Cref{ass:mckean_lipschitz_revised} is an important open problem. The subsequent results in this section are conditional upon this assumption holding.
\end{remark}

\begin{theorem}[Existence of Well-Behaved Mean-Field Neural Drivers]
\label{thm:mckean_driver_existence}
The class of neural drivers $g_\theta$ for which the implicitly defined volatility function $\nu_\theta(t,x,\mu)$ satisfies the Lipschitz regularity and growth conditions of \Cref{ass:mckean_lipschitz_revised} is non-empty.
\end{theorem}

\begin{proof}
The proof is constructive. We will define a specific architecture for the specialized driver $g_\theta(t,x,\mu,z)$ and show that by imposing architecturally enforceable constraints on its components, the resulting implicit volatility function $\nu_\theta$ satisfies the conditions of \Cref{ass:mckean_lipschitz_revised}.

\paragraph{Step 1: Architectural Design.}
We structure the specialized driver $g_\theta$ to separate the dependencies on $z$ and $\mu$:
\begin{equation}
\label{eq:mckean_driver_arch}
    g_\theta(t, x, \mu, z) \coloneqq h_{\theta_1}(t, x, z) - \mathbb{E}_{Y \sim \mu}[\phi_{\theta_2}(x, Y)],
\end{equation}
where $h_{\theta_1}: [0,T] \times \R \times \R \to \R$ and $\phi_{\theta_2}: \R \times \R \to \R$ are functions parameterized by distinct neural networks with parameters $\theta_1$ and $\theta_2$ respectively, with $\theta = (\theta_1, \theta_2)$. The expectation $\mathbb{E}_{Y \sim \mu}[\cdot]$ is equivalent to the integral $\int_\R \phi_{\theta_2}(x, y)\mu(\dd y)$.

\paragraph{Step 2: Defining the Implicit Volatility Function $\nu_\theta$.}
The canonical NBM condition is $g_\theta(t, x, \mu, z) = 0$. This yields the equation:
\[ h_{\theta_1}(t, x, z) = \mathbb{E}_{Y \sim \mu}[\phi_{\theta_2}(x, Y)]. \]
To ensure a unique, well-behaved root $z = \nu_\theta(t,x,\mu)$, we enforce structural properties on $h_{\theta_1}$. As in the proof of \Cref{thm:existence_of_theta}, we design the network for $h_{\theta_1}$ such that for any fixed $(t,x)$, the function $z \mapsto h_{\theta_1}(t,x,z)$ is strictly monotonic and surjective onto $\R$. This can be achieved, for example, by ensuring $\partial_z h_{\theta_1}$ is bounded below by a positive constant.

Under this condition, $h_{\theta_1}$ has a well-defined inverse with respect to its third argument, which we denote $h_{\theta_1}^{-1}(t,x,\cdot)$. The implicit volatility function is then given explicitly by:
\begin{equation}
\label{eq:nu_theta_explicit}
    \nu_\theta(t, x, \mu) \coloneqq h_{\theta_1}^{-1}\left(t, x, \mathbb{E}_{Y \sim \mu}[\phi_{\theta_2}(x, Y)]\right).
\end{equation}

\paragraph{Step 3: Imposing Enforceable Architectural Constraints.}
We now impose regularity conditions on the component networks $h_{\theta_1}$ and $\phi_{\theta_2}$. These conditions can be enforced during training using techniques like spectral normalization (to control Lipschitz constants) or by specific architectural choices (e.g., parameterizing weights as positive).
\begin{enumerate}[label=(C\arabic*), leftmargin=*]
    \item \textbf{Monotonicity and Regularity of $h_{\theta_1}$:} For all $(t,x,z)$, there exists a constant $c_h > 0$ such that $\partial_z h_{\theta_1}(t,x,z) \ge c_h$. This ensures the inverse $h_{\theta_1}^{-1}$ exists and, by the inverse function theorem, its partial derivative with respect to its third argument is bounded by $1/c_h$.
    \item \textbf{Lipschitzness of $h_{\theta_1}^{-1}$:} The inverse function $v(t,x,w) \mapsto h_{\theta_1}^{-1}(t,x,w)$ is Lipschitz in all its arguments on its domain. Let its Lipschitz constants be $L_{h,t}, L_{h,x}, L_{h,w}$.
    \item \textbf{Lipschitzness of $\phi_{\theta_2}$:} The function $\phi_{\theta_2}(x,y)$ is globally Lipschitz in both its arguments, with constants $L_{\phi,x}$ and $L_{\phi,y}$. That is, for any $x_1, x_2, y_1, y_2 \in \R$:
    \[ |\phi_{\theta_2}(x_1, y_1) - \phi_{\theta_2}(x_2, y_2)| \le L_{\phi,x}|x_1-x_2| + L_{\phi,y}|y_1-y_2|. \]
    \item \textbf{Growth of $\phi_{\theta_2}$:} $\phi_{\theta_2}$ has at most linear growth: there is a constant $C_\phi > 0$ such that $|\phi_{\theta_2}(x,y)| \le C_\phi(1+|x|+|y|)$.
\end{enumerate}

\paragraph{Step 4: Proof of Lipschitz Continuity in Measure (Condition ii).}
Let $(t,x)$ be fixed. We want to bound $|\nu_\theta(t, x, \mu_1) - \nu_\theta(t, x, \mu_2)|$ for $\mu_1, \mu_2 \in \Pcal_2(\R)$.
Let $I(\mu) \coloneqq \mathbb{E}_{Y \sim \mu}[\phi_{\theta_2}(x, Y)]$. From \eqref{eq:nu_theta_explicit}, we have:
\begin{align*}
    |\nu_\theta(t, x, \mu_1) - \nu_\theta(t, x, \mu_2)| &= |h_{\theta_1}^{-1}(t,x,I(\mu_1)) - h_{\theta_1}^{-1}(t,x,I(\mu_2))| \\
    &\le L_{h,w} |I(\mu_1) - I(\mu_2)| \quad \text{(by (C2))}.
\end{align*}
Now we bound $|I(\mu_1) - I(\mu_2)|$. By the definition of the 2-Wasserstein distance, there exists a coupling $(Y_1, Y_2)$ of random variables with marginal laws $\mu_1$ and $\mu_2$ respectively, such that $W_2(\mu_1, \mu_2)^2 = \E[|Y_1-Y_2|^2]$.
\begin{align*}
    |I(\mu_1) - I(\mu_2)| &= |\E_{Y_1 \sim \mu_1}[\phi_{\theta_2}(x, Y_1)] - \E_{Y_2 \sim \mu_2}[\phi_{\theta_2}(x, Y_2)]| \\
    &= |\E[\phi_{\theta_2}(x, Y_1) - \phi_{\theta_2}(x, Y_2)]| \quad \text{(using the coupling)} \\
    &\le \E[|\phi_{\theta_2}(x, Y_1) - \phi_{\theta_2}(x, Y_2)|] \quad \text{(by Jensen's inequality)} \\
    &\le \E[L_{\phi,y}|Y_1 - Y_2|] \quad \text{(by Lipschitzness of $\phi_{\theta_2}$ in $y$, (C3))} \\
    &\le L_{\phi,y} \left( \E[|Y_1 - Y_2|^2] \right)^{1/2} \quad \text{(by Cauchy-Schwarz)} \\
    &= L_{\phi,y} W_2(\mu_1, \mu_2).
\end{align*}
Combining the inequalities, we get:
\[ |\nu_\theta(t, x, \mu_1) - \nu_\theta(t, x, \mu_2)| \le (L_{h,w} L_{\phi,y}) W_2(\mu_1, \mu_2). \]
This establishes Lipschitz continuity in the measure with constant $L_\mu = L_{h,w} L_{\phi,y}$.

\paragraph{Step 5: Proof of Lipschitz Continuity in State (Condition i).}
Let $(t,\mu)$ be fixed. Let $I(x,\mu) \coloneqq \mathbb{E}_{Y \sim \mu}[\phi_{\theta_2}(x, Y)]$.
\begin{align*}
    &|\nu_\theta(t, x_1, \mu) - \nu_\theta(t, x_2, \mu)| = |h_{\theta_1}^{-1}(t,x_1,I(x_1,\mu)) - h_{\theta_1}^{-1}(t,x_2,I(x_2,\mu))| \\
    &\le L_{h,x}|x_1-x_2| + L_{h,w}|I(x_1,\mu) - I(x_2,\mu)| \quad \text{(by (C2))}.
\end{align*}
We bound the second term:
\begin{align*}
    |I(x_1,\mu) - I(x_2,\mu)| &= |\E_{Y \sim \mu}[\phi_{\theta_2}(x_1, Y) - \phi_{\theta_2}(x_2, Y)]| \\
    &\le \E_{Y \sim \mu}[|\phi_{\theta_2}(x_1, Y) - \phi_{\theta_2}(x_2, Y)|] \\
    &\le \E_{Y \sim \mu}[L_{\phi,x}|x_1 - x_2|] = L_{\phi,x}|x_1 - x_2| \quad \text{(by (C3))}.
\end{align*}
Substituting this back, we get:
\[ |\nu_\theta(t, x_1, \mu) - \nu_\theta(t, x_2, \mu)| \le (L_{h,x} + L_{h,w}L_{\phi,x})|x_1-x_2|. \]
This establishes Lipschitz continuity in the state with constant $L_x = L_{h,x} + L_{h,w}L_{\phi,x}$.

\paragraph{Step 6: Proof of Linear Growth (Condition iii).}
We need to bound $|\nu_\theta(t,x,\mu)|^2$. We assume without loss of generality that $h_{\theta_1}^{-1}(t,x,0)=0$ and $\phi_{\theta_2}(0,0)=0$.
\begin{align*}
    |\nu_\theta(t,x,\mu)| &= |h_{\theta_1}^{-1}(t,x,I(x,\mu)) - h_{\theta_1}^{-1}(t,x,0)| \\
    &\le L_{h,x}|x| + L_{h,w}|I(x,\mu)|.
\end{align*}
Now we bound $|I(x,\mu)|$ using the linear growth condition (C4) on $\phi_{\theta_2}$:
\begin{align*}
|I(x,\mu)| &= |\E_{Y \sim \mu}[\phi_{\theta_2}(x,Y)]| \le \E_{Y \sim \mu}[|\phi_{\theta_2}(x,Y)|] \\
&\le \E_{Y \sim \mu}[C_\phi(1+|x|+|Y|)] = C_\phi(1+|x|+\E_{Y \sim \mu}[|Y|]).
\end{align*}
By Cauchy-Schwarz, $\E[|Y|] \le (\E[|Y|^2])^{1/2} = (\int y^2 \mu(\dd y))^{1/2} = W_2(\mu,\delta_0)$.
So, $|I(x,\mu)| \le C_\phi(1+|x|+W_2(\mu,\delta_0))$. Substituting back:
\[ |\nu_\theta(t,x,\mu)| \le L_{h,x}|x| + L_{h,w}C_\phi(1+|x|+W_2(\mu,\delta_0)) \le C(1+|x|+W_2(\mu,\delta_0)), \]
for some constant $C$. Squaring this gives:
\[ |\nu_\theta(t,x,\mu)|^2 \le C^2(1+|x|+W_2(\mu,\delta_0))^2 \le 3C^2(1+|x|^2+W_2(\mu,\delta_0)^2), \]
which is the required linear growth condition.

Since we have constructively shown that there exists a family of neural network architectures for $g_\theta$ such that the resulting implicit volatility function $\nu_\theta$ satisfies all three conditions of \Cref{ass:mckean_lipschitz_revised}, the set of such drivers is non-empty. This completes the proof.
\end{proof}

\subsection{The Mean-Field Limit: A Non-Linear SDE}

We first establish the well-posedness of the limiting non-linear process, which is a solution to a stochastic differential equation whose coefficients depend on the law of the solution itself.

\begin{proposition}[Well-Posedness of the Mean-Field SDE]
\label{prop:mckean_sde_wellposedness}
Let \Cref{ass:mckean_lipschitz_revised} hold and let $\mu_0 \in \Pcal_2(\R)$ be a given initial law. Then there exists a unique probability measure flow $(\mu_t)_{t \in [0,T]} \in C([0,T], \Pcal_2(\R))$ such that $\mu_t$ is the law of the unique strong solution to the non-linear SDE:
\begin{equation} \label{eq:nlsde_revised}
    \dd M_t = \nu_\theta(t, M_t, \mu_t)\dd W_t, \quad \text{with } \mathrm{Law}(M_t) = \mu_t \text{ and } \mathrm{Law}(M_0) = \mu_0.
\end{equation}
\end{proposition}

\begin{proof}
The proof proceeds by constructing a contraction mapping on the complete metric space of measure-valued paths, $\mathcal{X} \coloneqq C([0,T], \Pcal_2(\R))$, equipped with the metric $d_T(\pi^1, \pi^2) \coloneqq \sup_{t\in[0,T]} W_2(\pi^1_t, \pi^2_t)$. A solution to \eqref{eq:nlsde_revised} is a fixed point of a suitable map on this space.

\textbf{Step 1: The Picard Iteration Map.}
Define a map $\Phi: \mathcal{X} \to \mathcal{X}$. For any given measure path $(\pi_t)_{t \in [0,T]} \in \mathcal{X}$, consider the classical SDE:
\[ \dd X_t = \nu_\theta(t, X_t, \pi_t) \dd W_t, \quad \mathrm{Law}(X_0) = \mu_0. \]
Under \Cref{ass:mckean_lipschitz_revised}, the coefficient $v(t,x) \coloneqq \nu_\theta(t,x,\pi_t)$ is Lipschitz in $x$ and satisfies a linear growth condition. Standard SDE theory guarantees the existence of a unique strong solution $X_t$. We define the map $\Phi$ as the law of this solution: $\Phi(\pi)_t \coloneqq \mathrm{Law}(X_t)$. Stability estimates for SDEs ensure that $t \mapsto \Phi(\pi)_t$ is continuous in the Wasserstein metric, so $\Phi(\pi) \in \mathcal{X}$.

\textbf{Step 2: Proving the Contraction Property.}
We show that for a sufficiently small time horizon $T_0 > 0$, $\Phi$ is a contraction. Let $(\pi^1_t), (\pi^2_t) \in \mathcal{X}$. Let $X^1, X^2$ be the corresponding solutions, coupled to be driven by the same Brownian motion $W$ and with the same initial condition $X^1_0 = X^2_0 \sim \mu_0$. The squared Wasserstein distance is bounded by the mean squared difference under this coupling:
\begin{align*}
W_2(\Phi(\pi^1)_t, \Phi(\pi^2)_t)^2 &\le \E[\abs{X^1_t - X^2_t}^2] \\
&= \E\left[ \abs{\int_0^t \left( \nu_\theta(s, X^1_s, \pi^1_s) - \nu_\theta(s, X^2_s, \pi^2_s) \right) \dd W_s}^2 \right] \\
&= \E\left[ \int_0^t \abs{\nu_\theta(s, X^1_s, \pi^1_s) - \nu_\theta(s, X^2_s, \pi^2_s)}^2 \dd s \right] \quad \text{(by Itô isometry)}.
\end{align*}
We bound the integrand using the triangle inequality and \Cref{ass:mckean_lipschitz_revised}:
\begin{align*}
&\abs{\nu_\theta(s, X^1_s, \pi^1_s) - \nu_\theta(s, X^2_s, \pi^2_s)}^2 \\
&\quad\le \left( \abs{\nu_\theta(s, X^1_s, \pi^1_s) - \nu_\theta(s, X^2_s, \pi^1_s)} + \abs{\nu_\theta(s, X^2_s, \pi^1_s) - \nu_\theta(s, X^2_s, \pi^2_s)} \right)^2 \\
&\quad\le 2\abs{\nu_\theta(s, X^1_s, \pi^1_s) - \nu_\theta(s, X^2_s, \pi^1_s)}^2 + 2\abs{\nu_\theta(s, X^2_s, \pi^1_s) - \nu_\theta(s, X^2_s, \pi^2_s)}^2 \\
&\quad\le 2L_x^2 \abs{X^1_s - X^2_s}^2 + 2L_\mu^2 W_2(\pi^1_s, \pi^2_s)^2.
\end{align*}
Let $e(t) \coloneqq \E[\abs{X^1_t - X^2_t}^2]$. Integrating the bound above from $0$ to $t$:
\begin{align*}
e(t) &\le \int_0^t \left( 2L_x^2 \E[\abs{X^1_s - X^2_s}^2] + 2L_\mu^2 W_2(\pi^1_s, \pi^2_s)^2 \right) \dd s \\
&\le 2L_x^2 \int_0^t e(s) \dd s + 2L_\mu^2 \int_0^t d_s(\pi^1, \pi^2)^2 \dd s \\
&\le 2L_x^2 \int_0^t e(s) \dd s + 2L_\mu^2 t \cdot d_t(\pi^1, \pi^2)^2.
\end{align*}
By Grönwall's inequality, $e(t) \le (2L_\mu^2 t \cdot d_t(\pi^1, \pi^2)^2) e^{2L_x^2 t}$. Taking the supremum over $t \in [0, T_0]$:
\[ d_{T_0}(\Phi(\pi^1), \Phi(\pi^2))^2 = \sup_{t \in [0,T_0]} W_2(\Phi(\pi^1)_t, \Phi(\pi^2)_t)^2 \le \sup_{t \in [0,T_0]} e(t) \le \left( 2L_\mu^2 T_0 e^{2L_x^2 T_0} \right) d_{T_0}(\pi^1, \pi^2)^2. \]
We can choose $T_0 > 0$ small enough such that the constant $K \coloneqq \sqrt{2L_\mu^2 T_0 e^{2L_x^2 T_0}} < 1$. For such a $T_0$, $\Phi$ is a contraction on the complete metric space $C([0,T_0], \Pcal_2(\R))$. The Banach fixed-point theorem guarantees the existence of a unique fixed point on $[0,T_0]$.

\textbf{Step 3: Extension to $[0,T]$.}
This local solution can be extended to the full interval $[0,T]$ by iterating the procedure. We solve on $[0, T_0]$, then use the terminal law $\mu_{T_0}$ as the initial law for the interval $[T_0, 2T_0]$, and so on. Since the contraction constant $K$ depends only on the length of the time interval, not on the initial condition, this procedure can be repeated a finite number of times to cover $[0,T]$.
\end{proof}

\subsection{Propagation of Chaos}

We can now state the main result on the convergence of the $N$-particle system.

\begin{proposition}[Propagation of Chaos for the NBM System]
\label{prop:propagation_of_chaos}
Let \Cref{ass:mckean_lipschitz_revised} hold. Let the initial conditions $\{m_0^i\}_{i=1}^N$ be i.i.d. samples from $\mu_0 \in \Pcal_2(\R)$. As $N \to \infty$, the empirical measure $\mu^N_t$ of the interacting system \eqref{eq:interacting_nbm_sde_revised} converges in law, in the space $C([0,T], \Pcal_2(\R))$, to the deterministic measure flow $(\mu_t)_{t\in[0,T]}$ which is the unique solution to the non-linear SDE \eqref{eq:nlsde_revised}.
\end{proposition}

\begin{proof}
Let $(\bar{M}^i_t)_{i=1}^N$ be an ideal system of $N$ i.i.d. processes, each solving the non-linear SDE \eqref{eq:nlsde_revised} independently, driven by the same independent Brownian motions $(W^i_t)$ and with the same initial conditions $\bar{M}^i_0 = m^i_0$. Let $\mu_t = \mathrm{Law}(\bar{M}^i_t)$ be the deterministic measure flow from \Cref{prop:mckean_sde_wellposedness}. We aim to bound the average mean-squared error between the true system and the ideal system: $E_N(t) \coloneqq \frac{1}{N} \sum_{i=1}^N \E[\abs{M^{i,N}_t - \bar{M}^i_t}^2]$.

Applying Itô isometry to the difference $M^{i,N}_t - \bar{M}^i_t$:
\begin{align*}
\E[\abs{M^{i,N}_t - \bar{M}^i_t}^2] &= \E\left[\int_0^t \abs{\nu_\theta(s, M^{i,N}_s, \mu^N_s) - \nu_\theta(s, \bar{M}^i_s, \mu_s)}^2 \dd s\right] \\
&\le 2\E\left[\int_0^t \left(L_x^2\abs{M^{i,N}_s - \bar{M}^i_s}^2 + L_\mu^2 W_2(\mu^N_s, \mu_s)^2 \right) \dd s\right].
\end{align*}
Averaging over $i=1,\dots,N$:
\[ E_N(t) \le 2L_x^2 \int_0^t E_N(s) \dd s + 2L_\mu^2 \int_0^t \E[W_2(\mu^N_s, \mu_s)^2] \dd s. \]
Let $\bar{\mu}^N_s = \frac{1}{N}\sum_i \delta_{\bar{M}^i_s}$ be the empirical measure of the ideal system. We bound the Wasserstein term by the triangle inequality:
\begin{align*}
\E[W_2(\mu^N_s, \mu_s)^2] &\le 2\E[W_2(\mu^N_s, \bar{\mu}^N_s)^2] + 2\E[W_2(\bar{\mu}^N_s, \mu_s)^2] \\
&\le 2\E\left[\frac{1}{N}\sum_i \abs{M^{i,N}_s - \bar{M}^i_s}^2\right] + 2\E[W_2(\bar{\mu}^N_s, \mu_s)^2] \\
&= 2 E_N(s) + 2\epsilon_N(s),
\end{align*}
where $\epsilon_N(s) \coloneqq \E[W_2(\bar{\mu}^N_s, \mu_s)^2]$ measures the convergence of the empirical measure of $N$ i.i.d. samples to their common law. It is a standard result that $\lim_{N\to\infty} \epsilon_N(s) = 0$ for each $s$, and one can find a uniform bound $\sup_{s \in [0,T]} \epsilon_N(s) \le C_T/N$ for some constant $C_T$ (see, e.g., \cite[Vol I]{CarmonaDelarue2018}).
Substituting this back into the inequality for $E_N(t)$:
\[ E_N(t) \le (2L_x^2 + 4L_\mu^2) \int_0^t E_N(s) \dd s + 4L_\mu^2 \int_0^t \epsilon_N(s) \dd s. \]
By Grönwall's inequality, there exists a constant $C'_{T}$ depending on $L_x, L_\mu, T$ such that:
\[ E_N(t) \le \left(4L_\mu^2 \int_0^T \epsilon_N(s) \dd s\right) e^{C'_{T} T}. \]
Since $\int_0^T \epsilon_N(s) \dd s \to 0$ as $N \to \infty$, we conclude that $\lim_{N \to \infty} \sup_{t \in [0,T]} E_N(t) = 0$. This proves convergence in mean square of any particle $M^{i,N}$ to the corresponding ideal particle $\bar{M}^i$. This implies convergence in law of the empirical measure $\mu^N_t$ to the deterministic measure $\mu_t$.
\end{proof}

\subsection{The Neural McKean-Vlasov PDE}

As a consequence of the limiting behavior, if the limiting measure flow admits a density, that density must solve a non-linear Fokker-Planck-Kolmogorov type equation.

\begin{proposition}[The Neural McKean-Vlasov Equation]
\label{cor:mckean_vlasov_pde}
Under the assumptions of \Cref{prop:mckean_sde_wellposedness}, if the measure flow $\mu_t$ admits a density $u(t,x)$ for all $t \in [0,T]$ such that $u \in C^{1,2}([0,T]\times\R)$, then $u(t,x)$ is a classical solution to the \textbf{Neural McKean-Vlasov Equation}:
\begin{equation}
\label{eq:mckean_vlasov_pde_revised}
    \partial_t u(t,x) = \frac{1}{2} \partial_{xx} \left[ \nu_\theta(t, x, \mu_t)^2 u(t,x) \right],
\end{equation}
with initial condition $u(0,x) = u_0(x)$, where $\mu_t$ is the measure with density $u(t, \cdot)$. In general, $u$ is a weak solution to this equation.
\end{proposition}

\begin{proof}
The derivation follows from considering the evolution of the expectation of a smooth test function $\phi \in C_c^\infty(\R)$ with respect to the law $\mu_t = \mathrm{Law}(M_t)$.
On one hand, $\frac{\dd}{\dd t} \E[\phi(M_t)] = \frac{\dd}{\dd t} \int_\R \phi(x) u(t,x) \dd x = \int_\R \phi(x) \partial_t u(t,x) \dd x$.
On the other hand, by Itô's formula, the infinitesimal generator of the process $M_t$ from \eqref{eq:nlsde_revised} acts on $\phi(x)$ as $\mathcal{L}_t \phi(x) = \frac{1}{2}\nu_\theta(t,x,\mu_t)^2 \phi''(x)$. Thus,
\begin{align*}
    \frac{\dd}{\dd t} \E[\phi(M_t)] &= \E[\mathcal{L}_t \phi(M_t)] = \int_\R \frac{1}{2}\nu_\theta(t,x,\mu_t)^2 \phi''(x) u(t,x) \dd x.
\end{align*}
Integrating the right-hand side by parts twice and using the compact support of $\phi$ to discard boundary terms:
\begin{align*}
\int_\R \frac{1}{2}\nu_\theta(t,x,\mu_t)^2 \phi''(x) u(t,x) \dd x &= - \int_\R \partial_x\left[\frac{1}{2}\nu_\theta(t,x,\mu_t)^2 u(t,x)\right] \phi'(x) \dd x \\
&= \int_\R \frac{1}{2} \partial_{xx}\left[\nu_\theta(t,x,\mu_t)^2 u(t,x)\right] \phi(x) \dd x.
\end{align*}
Equating the two expressions for $\frac{\dd}{\dd t}\E[\phi(M_t)]$ yields:
\[ \int_\R \phi(x) \left( \partial_t u(t,x) - \frac{1}{2} \partial_{xx}\left[\nu_\theta(t,x,\mu_t)^2 u(t,x)\right] \right) \dd x = 0. \]
Since this holds for all $\phi \in C_c^\infty(\R)$, the term in the parenthesis must be zero in the sense of distributions, which gives the weak form of equation \eqref{eq:mckean_vlasov_pde_revised}.
\end{proof}

\begin{remark}[Individual Ambiguity and Collective Dynamics]
This section provides a formal link between the microscopic axioms of the NBM and the macroscopic evolution of a large system. The non-linearity in the driver $g_\theta$, which on an individual level can be interpreted as defining a non-additive measure or capacity that reflects ambiguity, manifests at the collective level as a non-linear dependence in the diffusion coefficient of the governing Fokker-Planck equation. The interaction through the law $\mu_t$ is the mean-field consequence of each particle's volatility being shaped by the ambiguous environment created by all other particles.
\end{remark}

\section{A Universal Approximation Theorem for Canonical NBMs}
\label{sec:uat}

We now show that the NBM framework is sufficiently expressive to approximate any well-behaved standard diffusion process.

\begin{theorem}[Universal Approximation Property of Canonical NBMs]
\label{thm:nbm_uat}
Let $K \subset \R$ be a compact set, and let $\nu: [0,T] \times K \to (0, \infty)$ be any continuous function that is bounded away from zero, i.e., $\inf_{(t,x) \in [0,T] \times K} \nu(t,x) > 0$.
Then, for any $\epsilon > 0$, there exists a parameter vector $\theta$ and a specialized neural driver $g_\theta(t,x,z)$ satisfying \Cref{ass:implicit_vol} on $[0,T] \times K$, such that the resulting implicit volatility function $\nu_\theta(t,x)$ satisfies:
\[
    \sup_{(t,x) \in [0,T] \times K} \abs{\nu_\theta(t,x) - \nu(t,x)} < \epsilon.
\]
\end{theorem}

\begin{proof}
The proof is constructive. We will define an ideal target function whose root is the target volatility $\nu$. We will then use the universal approximation theorem for neural networks and their derivatives to construct a driver $g_\theta$ that is uniformly close to this ideal function. Finally, we show that the root of $g_\theta$ must be close to the root of the ideal function.

\textbf{Step 1: Define the Target Function and Compact Domain.}

Let the target volatility function be $\nu(t,x)$. Since $\nu$ is a continuous function on the compact set $[0,T] \times K$, it is bounded and attains its bounds. Let these be:
\[
c_\nu \coloneqq \inf_{(t,x) \in [0,T] \times K} \nu(t,x) \quad \text{and} \quad
C_\nu \coloneqq \sup_{(t,x) \in [0,T] \times K} \nu(t,x).
\]
By hypothesis, $0 < c_\nu \le C_\nu < \infty$.

We define an ideal driver function $h: [0,T] \times K \times \R \to \R$ whose root with respect to its third argument is precisely $\nu(t,x)$:
\[
h(t,x,z) \coloneqq z - \nu(t,x).
\]
The unique root of $h(t,x,z)=0$ for $z$ is $z=\nu(t,x)$. Furthermore, the partial derivative with respect to $z$ is $\partial_z h(t,x,z) = 1$, which is constant and non-zero.

Let $\epsilon > 0$ be the desired final approximation tolerance. We choose a parameter $\delta$ satisfying $0 < \delta < \min(\epsilon, c_\nu/2, 1/2)$. The rationale for this choice will become clear below. We define the compact domain $\mathcal{D} \subset [0,T] \times K \times \R$ on which our approximation will be performed:
\[
\mathcal{D} \coloneqq [0,T] \times K \times [c_\nu - \delta, C_\nu + \delta].
\]
Since $c_\nu - \delta > c_\nu - c_\nu/2 = c_\nu/2 > 0$, the $z$ component of this domain remains strictly positive.

\textbf{Step 2: Apply the $C^1$-Universal Approximation Theorem.}

The function $h(t,x,z)$ is of class $C^\infty$. We invoke the $C^1$-universal approximation theorem for neural networks (e.g., \cite{HornikStinchcombeWhite1990}), which states that a feedforward neural network with a single hidden layer and appropriate activation functions can approximate any $C^1$ function and its first derivatives uniformly on any compact set.
Therefore, for the compact domain $\mathcal{D}$ and the precision parameter $\delta$ defined above, there exists a parameter vector $\theta$ and a specialized neural driver $g_\theta(t,x,z)$ such that the following holds for all $(t,x,z) \in \mathcal{D}$:
\begin{align}
\abs{g_\theta(t,x,z) - h(t,x,z)} &< \delta \label{eq:uat_func_approx} \\
\abs{\partial_z g_\theta(t,x,z) - \partial_z h(t,x,z)} &< \delta \label{eq:uat_deriv_approx}
\end{align}
Substituting the definitions of $h$ and $\partial_z h$ gives:
\begin{align}
\abs{g_\theta(t,x,z) - (z - \nu(t,x))} &< \delta \label{eq:uat_func_explicit} \\
\abs{\partial_z g_\theta(t,x,z) - 1} &< \delta \label{eq:uat_deriv_explicit}
\end{align}

\textbf{Step 3: Verify Existence and Uniqueness of the Implicit Volatility Root.}

We now show that for each fixed $(t,x) \in [0,T] \times K$, the equation $g_\theta(t,x,z)=0$ has a unique root $z = \nu_\theta(t,x)$, and that this root lies within the interval $[c_\nu - \delta, C_\nu + \delta]$. This will verify conditions (a) and (b) of \Cref{ass:implicit_vol} on this domain.

First, from \eqref{eq:uat_deriv_explicit} and our choice of $\delta < 1/2$, we have:
\[
\partial_z g_\theta(t,x,z) > 1 - \delta > 1 - 1/2 = 1/2 > 0 \quad \text{for all } (t,x,z) \in \mathcal{D}.
\]
This implies that for any fixed $(t,x)$, the function $z \mapsto g_\theta(t,x,z)$ is strictly increasing on the interval $[c_\nu - \delta, C_\nu + \delta]$. A strictly monotonic function can have at most one root.

To prove existence, we evaluate $g_\theta$ at the boundaries of the interval $[\nu(t,x) - \delta, \nu(t,x) + \delta]$. Note that this interval is contained in $[c_\nu - \delta, C_\nu + \delta]$, so the approximation bounds hold.
\begin{itemize}
    \item At $z = \nu(t,x) - \delta$:
    The target function is $h(t,x, \nu(t,x) - \delta) = (\nu(t,x)-\delta) - \nu(t,x) = -\delta$.
    Using the approximation bound \eqref{eq:uat_func_approx}, we have $\abs{g_\theta(t,x, \nu(t,x)-\delta) - (-\delta)} < \delta$. This implies $-\delta < g_\theta(t,x, \nu(t,x)-\delta) + \delta < \delta$, which in turn implies $g_\theta(t,x, \nu(t,x)-\delta) < 0$.

    \item At $z = \nu(t,x) + \delta$:
    The target function is $h(t,x, \nu(t,x) + \delta) = (\nu(t,x)+\delta) - \nu(t,x) = \delta$.
    Using the approximation bound, we have $\abs{g_\theta(t,x, \nu(t,x)+\delta) - \delta} < \delta$. This implies $-\delta < g_\theta(t,x, \nu(t,x)+\delta) - \delta < \delta$, which in turn implies $g_\theta(t,x, \nu(t,x)+\delta) > 0$.
\end{itemize}
Since $z \mapsto g_\theta(t,x,z)$ is continuous and changes sign from negative to positive over the interval $(\nu(t,x) - \delta, \nu(t,x) + \delta)$, the Intermediate Value Theorem guarantees the existence of a root $\nu_\theta(t,x)$ within this interval. As the function is strictly monotonic on $\mathcal{D}$, this root is unique.

\textbf{Step 4: Establish the Uniform Error Bound for the Root.}

We have established that for each $(t,x)$, there is a unique root $\nu_\theta(t,x)$ such that $g_\theta(t,x, \nu_\theta(t,x)) = 0$.
From Step 3, we know that this root lies in the interval $(\nu(t,x)-\delta, \nu(t,x)+\delta)$, which immediately implies $\abs{\nu_\theta(t,x) - \nu(t,x)} < \delta$.

Since this holds for any $(t,x) \in [0,T] \times K$, we have shown that $\sup_{(t,x) \in [0,T] \times K} \abs{\nu_\theta(t,x) - \nu(t,x)} < \delta$.
By our initial choice of $\delta < \epsilon$, we have achieved the desired approximation accuracy:
\[
\sup_{(t,x) \in [0,T] \times K} \abs{\nu_\theta(t,x) - \nu(t,x)} < \epsilon.
\]
The constructed driver $g_\theta$ and its implicit volatility function $\nu_\theta$ satisfy the conditions of \Cref{ass:implicit_vol} on $[0,T] \times K$. Specifically, (a) unique root and (b) regularity are shown in Step 3, and (c) linear growth is trivially satisfied as $\nu_\theta$ is a continuous function on a compact set, hence bounded. This completes the proof.
\end{proof}

\begin{remark}[Implications of the Theorem]
This theorem establishes that the implicit formulation via the driver $g_\theta$ is a powerful re-parameterization rather than a restriction. It provides a theoretical guarantee that searching over the parameters $\theta$ of a suitably structured driver is effectively searching over a dense subset of all continuous diffusion models on compact domains.
\end{remark}

\section{A Volatility Selection Principle}
\label{sec:selection_principle}

The theory developed thus far has relied on the strong structural condition of \Cref{ass:implicit_vol}, which posits that the algebraic equation $g_\theta(t,x,z) = 0$ admits a unique positive root for the volatility $z$. While this assumption defines a tractable and non-empty class of models, it is restrictive. A neural driver $g_\theta$ learned from data, without architectural constraints enforcing monotonicity, may not be monotonic in $z$. For instance, $g_\theta$ could be a polynomial in $z$, leading to multiple positive roots.

This raises a fundamental question: if multiple volatility levels are consistent with the canonical NBM condition $g_\theta(t,M_t,\sigma_t)=0$, which volatility does the process select? The representation theorem (\Cref{thm:nbm_representation}) breaks down, as the function $\nu_\theta$ is no longer well-defined. To address this, we propose a variational selection principle, inspired by concepts in risk-sensitive control and physics, where the system endogenously selects the volatility that minimizes an associated potential function.

\begin{assumption}[Regular Multiple Roots]
\label{ass:multiple_roots}
For a given $\theta$, the specialized driver $g_\theta: [0,T] \times \R \times \R \to \R$ is of class $C^1$. For each $(t,x) \in [0,T] \times \R$, the set of positive roots
\[ \mathcal{Z}_\theta(t,x) \coloneqq \{ z \in (0, \infty) \mid g_\theta(t,x,z)=0 \} \]
is non-empty, finite, and contains only regular roots (i.e., $\partial_z g_\theta(t,x,z) \neq 0$ for all $z \in \mathcal{Z}_\theta(t,x)$).
\end{assumption}

To resolve the ambiguity of multiple roots, we introduce a potential function derived from the driver itself.

\begin{definition}[Volatility Potential]
\label{def:vol_potential}
The \textbf{volatility potential} associated with the specialized driver $g_\theta$ is the function $G_\theta: [0,T] \times \R \times \R \to \R$ defined by the integral:
\[ G_\theta(t,x,z) \coloneqq \int_0^z g_\theta(t,x,u) \dd u. \]
\end{definition}

Our proposed selection principle states that the system will adopt the volatility that corresponds to the lowest potential energy among all possible choices.

\begin{definition}[Selected Volatility via Minimal Potential]
\label{def:selected_vol}
The \textbf{selected volatility function} $\nu_\theta: [0,T] \times \R \to (0, \infty)$ is defined by the \textbf{minimal potential principle}:
\[ \nu_\theta(t,x) \coloneqq \underset{z \in \mathcal{Z}_\theta(t,x)}{\arg\min} \, G_\theta(t,x,z). \]
\end{definition}

For this definition to yield a well-behaved SDE, the resulting function $\nu_\theta(t,x)$ must be well-defined and sufficiently regular. The following assumption provides the precise conditions needed to guarantee this. It replaces the previous, less formal \Cref{ass:strict_minimizer}.

\begin{assumption}[Strict Minimal Potential and Branch Separation]
\label{ass:strict_minimizer}
The specialized driver $g_\theta$ and its associated potential $G_\theta$ satisfy the following for all $(t,x) \in [0,T] \times \R$:
\begin{enumerate}[label=(\roman*)]
    \item \textbf{Unique Global Minimizer:} The minimizer in \Cref{def:selected_vol} is unique. Let this be $\nu_\theta(t,x)$.
    \item \textbf{Strict Second-Order Condition:} The selected root corresponds to a strict local minimum of the potential, meaning the second-order condition holds strictly:
    \[ \frac{\partial^2 G_\theta}{\partial z^2}(t,x,\nu_\theta(t,x)) = \frac{\partial g_\theta}{\partial z}(t,x,\nu_\theta(t,x)) > 0. \]
    \item \textbf{Branch Separation:} The value of the potential at the selected root is strictly lower than at any other root:
    \[ G_\theta(t,x, \nu_\theta(t,x)) < G_\theta(t,x, z) \quad \forall z \in \mathcal{Z}_\theta(t,x), \, z \neq \nu_\theta(t,x). \]
    \item \textbf{Global Linear Growth:} The resulting selected volatility function $\nu_\theta(t,x)$ satisfies a linear growth condition: there exists a constant $C>0$ such that $|\nu_\theta(t,x)| \le C(1+|x|)$.
\end{enumerate}
\end{assumption}

Under this refined set of assumptions, we can provide a complete proof of the representation theorem.

\begin{theorem}[Representation with a Volatility Selection Principle]
\label{thm:nbm_representation_selection}
Let the driver $f_\theta$ satisfy the BSDE well-posedness conditions (\Cref{ass:wellposedness_multidim}), and let its specialization $g_\theta$ satisfy the multiple root conditions of \Cref{ass:multiple_roots} and the strict minimal potential conditions of \Cref{ass:strict_minimizer}.

Then, a canonical Neural-Brownian Motion consistent with the minimal potential principle exists and is the unique strong solution to the stochastic differential equation:
\[
    \dd M_t = \nu_\theta(t, M_t) \dd W_t, \quad M_0 = 0,
\]
where $\nu_\theta(t,x)$ is the selected volatility function defined in \Cref{def:selected_vol}.
\end{theorem}

\begin{proof}
The proof consists of three main parts. First, and most critically, we establish the regularity of the selected volatility function $\nu_\theta(t,x)$. Second, we use this regularity to prove the well-posedness of the corresponding SDE. Finally, we demonstrate the equivalence between solutions of this SDE and the definition of a canonical NBM under the selection principle.

\textbf{Step 1: Regularity of the Selected Volatility Function $\nu_\theta(t,x)$.}
The main challenge is to show that $\nu_\theta(t,x)$, defined via an $\arg\min$ operation, is a regular function. We will prove that $\nu_\theta$ is continuously differentiable ($C^1$), which implies it is locally Lipschitz.

Let an arbitrary point $(t_0, x_0) \in [0,T] \times \R$ be given. Let $z_0 \coloneqq \nu_\theta(t_0, x_0)$. By the definition of $\nu_\theta$ and the assumptions, we know two facts about the point $(t_0, x_0, z_0)$:
\begin{enumerate}
    \item $z_0$ is a root of $g_\theta$: $g_\theta(t_0, x_0, z_0) = 0$.
    \item $z_0$ satisfies the strict second-order condition from \Cref{ass:strict_minimizer}(ii): $\partial_z g_\theta(t_0, x_0, z_0) > 0$.
\end{enumerate}
These are precisely the conditions required to apply the Implicit Function Theorem. Let us define the function $F: [0,T] \times \R \times (0,\infty) \to \R$ by $F(t,x,z) \coloneqq g_\theta(t,x,z)$. We have shown that $F(t_0, x_0, z_0) = 0$ and its partial derivative with respect to its third variable, $\partial_z F(t_0, x_0, z_0)$, is non-zero.

By the Implicit Function Theorem, there exist an open neighborhood $U$ of $(t_0, x_0)$ and a unique continuously differentiable function $\psi: U \to (0,\infty)$ such that:
\begin{itemize}
    \item $\psi(t_0, x_0) = z_0$.
    \item $g_\theta(t, x, \psi(t,x)) = 0$ for all $(t,x) \in U$.
\end{itemize}
This function $\psi(t,x)$ represents a local, smooth branch of roots of $g_\theta$. We must now show that our globally defined function $\nu_\theta(t,x)$ coincides with this smooth local function $\psi(t,x)$ on the neighborhood $U$.

By \Cref{ass:strict_minimizer}(iii), the separation condition holds at $(t_0, x_0)$: $G_\theta(t_0, x_0, z_0) < G_\theta(t_0, x_0, z')$ for all other roots $z' \in \mathcal{Z}_\theta(t_0, x_0)$. Since $g_\theta$ is $C^1$, the potential $G_\theta$ is $C^2$. Furthermore, the set of roots $\mathcal{Z}_\theta(t,x)$ is finite. The locations of the roots and the values of $G_\theta$ at these roots vary continuously with $(t,x)$. Therefore, due to the strict inequality at $(t_0, x_0)$, there exists a (possibly smaller) neighborhood $V \subseteq U$ of $(t_0, x_0)$ such that for all $(t,x) \in V$, the root corresponding to the branch $\psi(t,x)$ remains the unique global minimizer of the potential $G_\theta(t,x, \cdot)$ over the set of roots $\mathcal{Z}_\theta(t,x)$.

This means that for all $(t,x) \in V$, we have $\nu_\theta(t,x) = \psi(t,x)$. Since $\psi$ is $C^1$ on $V$, $\nu_\theta$ is also $C^1$ on $V$. As the point $(t_0, x_0)$ was arbitrary, we conclude that the selected volatility function $\nu_\theta(t,x)$ is of class $C^1$ on its entire domain $[0,T] \times \R$. A $C^1$ function on a domain in $\R^n$ is necessarily locally Lipschitz.

\textbf{Step 2: Well-Posedness of the SDE.}
We have established that the function $\nu_\theta(t,x)$ is locally Lipschitz in its arguments $(t,x)$, and therefore also in $x$ for fixed $t$. Additionally, \Cref{ass:strict_minimizer}(iv) provides the global linear growth condition: $|\nu_\theta(t,x)| \le C(1+|x|)$.
A standard result in SDE theory (e.g., \cite{karatzas2012brownian}) states that an SDE of the form $\dd M_t = \nu(t, M_t) \dd W_t$ with initial condition $M_0=0$ has a unique strong solution, provided the coefficient function $\nu(t,x)$ is locally Lipschitz in $x$ and satisfies a linear growth condition. Our function $\nu_\theta(t,x)$ meets these requirements. Therefore, the SDE $\dd M_t = \nu_\theta(t, M_t) \dd W_t$ is well-posed and admits a unique, pathwise continuous strong solution.

\textbf{Step 3: Equivalence to Canonical NBM with Selection Principle.}
We now show this unique SDE solution is precisely the canonical NBM we seek.
\begin{itemize}
    \item[\textbf{($\impliedby$) Sufficiency:}] Let $(M_t)_{t \ge 0}$ be the unique strong solution to the SDE. We must verify that it is a canonical NBM satisfying the minimal potential principle.
    \begin{enumerate}
        \item By the SDE definition, $M_0=0$ and the paths are continuous.
        \item The SDE has no drift term, so the drift of $M_t$ under the physical measure $\Prob$ is $b_t=0$. This satisfies the canonical property.
        \item The volatility of the process is $\sigma_t = \nu_\theta(t, M_t)$. By the very definition of the function $\nu_\theta$ (\Cref{def:selected_vol}), its value is a root of $g_\theta$. Thus, $g_\theta(t, M_t, \sigma_t) = g_\theta(t, M_t, \nu_\theta(t, M_t)) = 0$. By \Cref{prop:drift_characterization_revised}, the drift $b_t$ must satisfy $b_t = -g_\theta(t, M_t, \sigma_t)$. This becomes $0 = -0$, which is true. Therefore, $M_t$ satisfies the $\cEneutheta$-martingale property.
    \end{enumerate}
    Thus, $M_t$ is a canonical NBM. By construction, its volatility follows the minimal potential principle.

    \item[\textbf{($\implies$) Necessity:}] Let $(M_t)_{t \ge 0}$ be a process that is, by hypothesis, a canonical NBM whose volatility is determined by the minimal potential principle.
    \begin{enumerate}
        \item As a canonical NBM, its Itô decomposition $\dd M_t = b_t \dd t + \sigma_t \dd W_t$ must have $b_t = 0$.
        \item As an $\cEneutheta$-martingale, its coefficients must satisfy $b_t = -g_\theta(t, M_t, \sigma_t)$. Combining these gives the constraint $g_\theta(t, M_t, \sigma_t) = 0$. So its volatility $\sigma_t$ must be one of the roots in $\mathcal{Z}_\theta(t, M_t)$.
        \item The additional hypothesis is that the volatility is selected by the minimal potential principle. This means that among all possible roots, the process must adopt the one that minimizes the potential $G_\theta$. By definition, this is precisely $\nu_\theta(t, M_t)$. Therefore, we must have $\sigma_t = \nu_\theta(t, M_t)$.
    \end{enumerate}
    Substituting these derived coefficients ($b_t=0, \sigma_t=\nu_\theta(t, M_t)$) back into the general Itô decomposition, we find that the process $M_t$ must satisfy the SDE $\dd M_t = \nu_\theta(t, M_t) \dd W_t$.
\end{itemize}
Since the SDE has a unique solution, the canonical NBM satisfying the selection principle is uniquely defined and given by this solution. This completes the proof.
\end{proof}

\section{Application: A Consistent Implicit Volatility Model for Option Pricing}
\label{sec:application}

We now construct a novel Implicit Volatility Model for option pricing that is a direct and rigorous application of the Neural-Brownian Motion theory. We postulate that the fundamental martingale in the risk-neutral world, the discounted asset price, is a canonical NBM.

\subsection{The Neural Implicit Volatility Model}

We work under a risk-neutral probability measure $\Qbb$, where the risk-free rate is $r$. The fundamental theorem of asset pricing states that the discounted price of any non-dividend-paying asset is a $\Qbb$-martingale. Our central modeling postulate is that this martingale is a canonical NBM.

\begin{axiom}[Postulate of the Neural Implicit Volatility Model]
\label{ax:niv_model}
Let $S_t$ be the price of a financial asset. Its discounted price process, $M_t \coloneqq e^{-rt}S_t$, is a one-dimensional canonical Neural-Brownian Motion under the risk-neutral measure $\Qbb$.
\end{axiom}

This single axiom, combined with our core theory, fully specifies the risk-neutral dynamics. Since $M_t$ is a canonical NBM under $\Qbb$, by \Cref{thm:nbm_representation} it must be the unique solution to an SDE of the form $\dd M_t = \nu_\theta(t, M_t) \dd W_t^\Qbb$, where $W_t^\Qbb$ is a $\Qbb$-Brownian motion and the volatility function $\nu_\theta$ is implicitly defined by the constraint $g_\theta(t, M_t, \nu_\theta(t, M_t)) = 0$ for a specialized neural driver $g_\theta$.

From this, we can derive the dynamics of the asset price $S_t = e^{rt}M_t$ itself. Using Itô's product rule:
\begin{align*}
    \dd S_t &= \dd(e^{rt}M_t) = (r e^{rt} \dd t) M_t + e^{rt} \dd M_t \\
    &= r (e^{rt} M_t) \dd t + e^{rt} \left( \nu_\theta(t, M_t) \dd W_t^\Qbb \right) \\
    &= r S_t \dd t + e^{rt} \nu_\theta(t, e^{-rt}S_t) \dd W_t^\Qbb.
\end{align*}
This gives the SDE for the asset price under $\Qbb$:
\begin{equation}
\label{eq:neural_gbm}
    \dd S_t = r S_t \dd t + \sigma_\theta(t, S_t) S_t \dd W_t^\Qbb,
\end{equation}
where the standard percentage volatility function (implied volatility) $\sigma_\theta(t,S)$ is defined as:
\begin{equation}
\label{eq:implicit_vol_finance}
    \sigma_\theta(t, S) \coloneqq \frac{e^{rt}\nu_\theta(t, e^{-rt}S)}{S}.
\end{equation}
Here, $\nu_\theta(t,m)$ is the function giving the root of $g_\theta(t,m,z)=0$. This formulation is now fully consistent with the theory of canonical NBMs. The model learns a relationship $g_\theta$ for the fundamental martingale $M_t$, which in turn implies a specific, derived local volatility structure for the asset price $S_t$.

\subsection{The Pricing Partial Differential Equation}
Given this model, the price of any European-style derivative must satisfy a corresponding PDE.

\begin{theorem}[The Neural Black-Scholes PDE]
\label{thm:neural_bs_pde}
Under \Cref{ax:niv_model}, the price $C(t,S)$ of a European derivative with payoff $H(S_T)$ satisfies the linear PDE:
\begin{equation}
\label{eq:neural_bs_pde_main}
    \frac{\partial C}{\partial t} + rS \frac{\partial C}{\partial S} + \frac{1}{2}(\sigma_\theta(t, S)S)^2 \frac{\partial^2 C}{\partial S^2} - rC = 0,
\end{equation}
for $(t,S) \in [0,T) \times \R^+$, subject to the terminal condition $C(T,S) = H(S)$. The function $\sigma_\theta(t,S)$ is the implicit volatility from \Cref{eq:implicit_vol_finance}.
\end{theorem}

\begin{proof}
By the Feynman-Kac formula, the price of the derivative at time $t$ is given by the risk-neutral expectation $C(t,S_t) = \E_\Qbb[e^{-r(T-t)}H(S_T) | \Fcal_t]$. The PDE \eqref{eq:neural_bs_pde_main} is the backward Kolmogorov equation associated with the SDE \eqref{eq:neural_gbm}. Alternatively, a no-arbitrage argument shows that the discounted price process $\tilde{C}_t = e^{-rt}C(t,S_t)$ must be a $\Qbb$-martingale. Applying Itô's formula to $\tilde{C}_t$ with the dynamics from \eqref{eq:neural_gbm}, the drift of $\tilde{C}_t$ is found to be $e^{-rt}$ times the left-hand side of \Cref{eq:neural_bs_pde_main}. Setting this drift to zero for the martingale property to hold yields the desired PDE.
\end{proof}

\subsection{Calibration: Learning the Market's Risk-Neutral View}
The model's parameters $\theta$ are learned by calibrating to a set of observed market prices $\{C^{\text{mkt}}_i\}$ for options with strikes $\{K_i\}$ and maturities $\{T_i\}$. This involves solving the optimization problem:
\begin{equation}
\label{eq:calibration_loss}
    \theta^* = \arg\min_{\theta \in \Theta} \sum_i w_i \left( C(t_0, S_0; K_i, T_i; \theta) - C^{\text{mkt}}_i \right)^2,
\end{equation}
where $C(\cdot; \theta)$ is the price obtained by solving the PDE \eqref{eq:neural_bs_pde_main} for a given parameter vector $\theta$. This process allows for the data-driven discovery of the specialized driver $g_{\theta^*}$ that best represents the market's consensus risk-neutral view, providing a principled and flexible method for reverse-engineering the dynamics implied by derivative prices. We note that this optimization problem is computationally demanding, as each evaluation of the objective function for a given $\theta$ requires solving the PDE \eqref{eq:neural_bs_pde_main} to price the entire set of calibration instruments.

\section{Conclusion}
This paper has introduced the Neural-Brownian Motion, a canonical stochastic process for worlds of learned ambiguity, built upon the axiomatic foundation of non-linear expectation. We have established its existence and uniqueness for a tractable class of models via a representation theorem, connecting abstract axioms to a concrete SDE whose volatility is implicitly defined by a neural network. The development of its stochastic calculus, including the Girsanov-type interpretation, provides essential tools for analysis. 

The successful universal approximation theorem guarantees the framework's expressiveness, while the formal analysis of its mean-field limit opens a clear and rigorous path toward understanding the collective behavior of systems governed by learned ambiguity. The application to finance demonstrates that the NBM is not merely a theoretical curiosity but a powerful and consistent modeling paradigm. Key avenues for future research include a full proof of the regularity conditions required for the mean-field analysis and a deeper exploration of the volatility selection principle for more general drivers.

\bibliographystyle{imsart-nameyear} 
\bibliography{reference}

\begin{thebibliography}{10}

\bibitem[\protect\citeauthoryear{Briand and Hu}{2008}]{BriandHu2008}
\begin{barticle}[author]
\bauthor{\bsnm{Briand},~\bfnm{Philippe}\binits{P.}} \AND \bauthor{\bsnm{Hu},~\bfnm{Ying}\binits{Y.}}
(\byear{2008}).
\btitle{Quadratic {BSDEs} with convex generators and unbounded terminal conditions}.
\bjournal{Probability Theory and Related Fields}
\bvolume{141}
\bpages{543--567}.
\end{barticle}
\endbibitem

\bibitem[\protect\citeauthoryear{Briand et~al.}{2003}]{BriandDelyonHuPardouxStoica2003}
\begin{barticle}[author]
\bauthor{\bsnm{Briand},~\bfnm{Philippe}\binits{P.}}, \bauthor{\bsnm{Delyon},~\bfnm{Bernard}\binits{B.}}, \bauthor{\bsnm{Hu},~\bfnm{Ying}\binits{Y.}}, \bauthor{\bsnm{Pardoux},~\bfnm{Etienne}\binits{E.}} \AND \bauthor{\bsnm{Stoica},~\bfnm{Lucian}\binits{L.}}
(\byear{2003}).
\btitle{{$L^p$} solutions of backward stochastic differential equations}.
\bjournal{Stochastic Processes and their Applications}
\bvolume{108}
\bpages{109--129}.
\end{barticle}
\endbibitem

\bibitem[\protect\citeauthoryear{Carmona and Delarue}{2018}]{CarmonaDelarue2018}
\begin{bbook}[author]
\bauthor{\bsnm{Carmona},~\bfnm{Ren{\'e}}\binits{R.}} \AND \bauthor{\bsnm{Delarue},~\bfnm{Fran{\c{c}}ois}\binits{F.}}
(\byear{2018}).
\btitle{Probabilistic Theory of Mean Field Games with Applications I-II}.
\bseries{Probability Theory and Stochastic Modelling}
\bvolume{83-84}.
\bpublisher{Springer}.
\bdoi{10.1007/978-3-319-58920-6}
\end{bbook}
\endbibitem

\bibitem[\protect\citeauthoryear{El~Karoui, Peng and Quenez}{1997}]{ElKarouiPengQuenez1997}
\begin{barticle}[author]
\bauthor{\bsnm{El~Karoui},~\bfnm{Nicole}\binits{N.}}, \bauthor{\bsnm{Peng},~\bfnm{Shige}\binits{S.}} \AND \bauthor{\bsnm{Quenez},~\bfnm{Marie~Claire}\binits{M.~C.}}
(\byear{1997}).
\btitle{Backward stochastic differential equations in finance}.
\bjournal{Mathematical Finance}
\bvolume{7}
\bpages{1--71}.
\end{barticle}
\endbibitem

\bibitem[\protect\citeauthoryear{Hornik, Stinchcombe and White}{1990}]{HornikStinchcombeWhite1990}
\begin{barticle}[author]
\bauthor{\bsnm{Hornik},~\bfnm{Kurt}\binits{K.}}, \bauthor{\bsnm{Stinchcombe},~\bfnm{Maxwell}\binits{M.}} \AND \bauthor{\bsnm{White},~\bfnm{Halbert}\binits{H.}}
(\byear{1990}).
\btitle{Universal approximation of an unknown mapping and its derivatives using multilayer feedforward networks}.
\bjournal{Neural Networks}
\bvolume{3}
\bpages{551--560}.
\end{barticle}
\endbibitem

\bibitem[\protect\citeauthoryear{Karatzas and Shreve}{2012}]{karatzas2012brownian}
\begin{bbook}[author]
\bauthor{\bsnm{Karatzas},~\bfnm{Ioannis}\binits{I.}} \AND \bauthor{\bsnm{Shreve},~\bfnm{Steven}\binits{S.}}
(\byear{2012}).
\btitle{Brownian motion and stochastic calculus}
\bvolume{113}.
\bpublisher{Springer Science \& Business Media}.
\end{bbook}
\endbibitem

\bibitem[\protect\citeauthoryear{Kobylanski}{2000}]{Kobylanski2000}
\begin{barticle}[author]
\bauthor{\bsnm{Kobylanski},~\bfnm{Magdalena}\binits{M.}}
(\byear{2000}).
\btitle{Backward stochastic differential equations and partial differential equations with quadratic growth}.
\bjournal{The Annals of Probability}
\bvolume{28}
\bpages{558--602}.
\end{barticle}
\endbibitem

\bibitem[\protect\citeauthoryear{Pardoux and Peng}{1990}]{PardouxPeng1990}
\begin{barticle}[author]
\bauthor{\bsnm{Pardoux},~\bfnm{Etienne}\binits{E.}} \AND \bauthor{\bsnm{Peng},~\bfnm{Shige}\binits{S.}}
(\byear{1990}).
\btitle{Adapted solution of a backward stochastic differential equation}.
\bjournal{Systems \& Control Letters}
\bvolume{14}
\bpages{55--61}.
\end{barticle}
\endbibitem

\bibitem[\protect\citeauthoryear{Qi}{2025}]{qi2025neuralexpectationoperators}
\begin{bmisc}[author]
\bauthor{\bsnm{Qi},~\bfnm{Qian}\binits{Q.}}
(\byear{2025}).
\btitle{Neural Expectation Operators}.
\end{bmisc}
\endbibitem

\bibitem[\protect\citeauthoryear{Sznitman}{1991}]{Sznitman1991}
\begin{bincollection}[author]
\bauthor{\bsnm{Sznitman},~\bfnm{Alain-Sol}\binits{A.-S.}}
(\byear{1991}).
\btitle{Topics in propagation of chaos}.
In \bbooktitle{\'Ecole d'\'Et\'e de Probabilit\'es de Saint-Flour XIX—1989},
(\beditor{\bfnm{P.~L.}\binits{P.~L.}~\bsnm{Henquin}}, ed.).
\bseries{Lecture Notes in Mathematics}
\bvolume{1464}
\bpages{165--251}.
\bpublisher{Springer Berlin Heidelberg}.
\bdoi{10.1007/BFb0085169}
\end{bincollection}
\endbibitem

\end{thebibliography}

\end{document}